\documentclass[onefignum,onetabnum]{siamonline250211}

\usepackage{amsfonts}
\usepackage{graphicx}

\usepackage{booktabs}
\usepackage{amssymb}
\usepackage{mathtools}
\usepackage{bm}
\usepackage{stackengine}

\usepackage{xcolor,tikz}
\usetikzlibrary{decorations.pathreplacing,calligraphy}
\usepackage{nicematrix}

\usepackage{cite} 

\usepackage{diagbox}
\usepackage{multirow}
\usepackage{siunitx}
\usepackage{graphicx}
\usepackage{wrapfig}

\usepackage{algorithm}
\usepackage[noEnd=true]{algpseudocodex}

\makeatletter
\def\input@path{{./Tables}{./Figures}{./Figures_lattice}{./Temp_Figures}}
\makeatother

\allowdisplaybreaks

\usepackage{enumitem}
\setlist[enumerate]{leftmargin=.5in}
\setlist[itemize]{leftmargin=.5in}

\newsiamremark{remark}{Remark}
\newsiamremark{hypothesis}{Hypothesis}
\crefname{hypothesis}{Hypothesis}{Hypotheses}
\newsiamthm{claim}{Claim}
\newsiamthm{problem}{Problem}
\newsiamremark{example}{Example}

\newcommand{\exref}[1]{\hyperref[#1]{Example~\ref*{#1}}}
\newcommand{\apref}[1]{\hyperref[#1]{Appendix~\ref*{#1}}}

\headers{Integer Image Recovery from Minimal DFT Samples}{H. W. Levinson and I. Viviano}

\title{Recovery of Integer Images from Minimal DFT Samples: Uniqueness and Inversion Algorithms\thanks{Submitted to the editors April 15, 2026.
\funding{This work was funded by the NSF, Grant 2513653}}}

\author{Howard W.~Levinson\thanks{Department of Computer Science, Oberlin College, Oberlin, OH 44074 
  (\email{hlevinso@oberlin.edu}).}\and
Isaac Viviano\thanks{Oberlin College, Oberlin, OH 44074 
  (\email{iviviano@oberlin.edu}}).}

\makeatletter
\newcommand*{\addFileDependency}[1]{
  \typeout{(#1)}
  \@addtofilelist{#1}
  \IfFileExists{#1}{}{\typeout{No file #1.}}
}
\makeatother

\DeclarePairedDelimiterX\set[1]\lbrace\rbrace{\def\given{\;\delimsize\vert\;}#1}
\DeclarePairedDelimiterX\event[1][]{\def\given{\;\delimsize\vert\;}#1}
\DeclarePairedDelimiter{\abs}{\lvert}{\rvert}
\DeclarePairedDelimiter{\norm}{\lVert}{\rVert}

\DeclarePairedDelimiter{\pareno}{(}{)} 
\NewDocumentCommand{\bigo}{s o m}{%
  \mathcal{O}\IfBooleanTF{#1}
    {\pareno*{#3}}%
    {\IfValueTF{#2}{\pareno[#2]{#3}}{\pareno{#3}}}%
}

\DeclareMathOperator*{\lcm}{lcm}

\newcommand{\R}{\mathbb{R}}
\newcommand{\C}{\mathbb{C}}
\newcommand{\Z}{\mathbb{Z}}

\newcommand{\cu}{\mathrm i}

\newcommand{\divs}{\operatorname{\mid}}
\newcommand{\ndivs}{\operatorname{\nmid}}
\newcommand{\eset}{\varnothing}
\newcommand{\uroot}[2]{\eta_{#1}^{#2}}
\newcommand{\idmat}[1]{{\tt I}_{#1}}
\newcommand{\e}[1]{e^{#1}}

\stackMath

\newcommand{\sumstack}[4]{
\underset{
    {#2}~({\rm mod}~{#3})
}{
    {\sum_{#1}~~~{#4}}
}
}

\graphicspath{{./Figures_image}}

\ifpdf
\hypersetup{
  pdftitle={Recovery of Integer Images from Minimal DFT Samples: Uniqueness and Inversion Algorithms},
  pdfauthor={H. W. Levinson, I. Viviano}
}
\fi

\begin{document}

\maketitle

\begin{abstract}
Exact reconstruction of an image from measurements of its Discrete Fourier Transform (DFT) typically requires all DFT coefficients to be available.
However, incorporating the prior assumption that the image contains only integer values enables unique recovery from a limited subset of DFT coefficients. 
This paper develops both theoretical and algorithmic foundations for this problem. 
We use algebraic properties of the DFT to define a reduction from two-dimensional recovery to several well-chosen one-dimensional recoveries.
Our reduction framework characterizes the minimum number and location of DFT coefficients that must be sampled to guarantee unique reconstruction of an integer-valued image. 
Algorithmically, we develop reconstruction procedures which use dynamic programming to efficiently recover an integer signal or image from its minimal set of DFT measurements. 
While the new inversion algorithms still involve NP-hard subproblems,
we demonstrate how the divide-and-conquer approach drastically reduces the associated search space.
To solve the NP-hard subproblems, we employ a lattice-based framework which leverages the LLL approximation algorithm to make the algorithms fast and practical.
\end{abstract}



\section{Introduction} 
Reconstructing a vector or matrix from limited discrete Fourier transform (DFT) measurements is a classic problem in signal processing and imaging~\cite{candes2006stable}.  
This challenge arises in many applications where noise corrupts parts of the data, or measuring the full spectrum is either impractical or prohibitively expensive. 
Whether or not we control the sampling strategy,
successful recovery from a limited spectrum often requires some form of {\it a priori} knowledge of the underlying signal or image.
For example, sparsity constraints have revolutionized the field through the framework of compressed sensing, enabling dramatic reductions in the number of required measurements, along with corresponding fast reconstruction algorithms~\cite{donoho2006compressed,duarte2011structured,gilbert2002near,gilbert2005improved,rudelson2008sparse,rudelson2006sparse,candes2006robust}.  
Related approaches study optimal sampling for sparse signals and can be used to develop sparse fast Fourier transform methods which are faster than their standard counterparts~\cite{plonka2016deterministic,rajaby2022structured,bittens2019deterministic,hsieh2013sparse,gilbert2014recent}.

While sparsity has dominated much of this research, another type of prior information has also been studied: that the unknown signal or image is restricted to a small set of discrete values~\cite{herman2012discrete,boufounos20081}.  
This line of work often leverages the algebraic structure of the DFT in combination with the discrete nature of the signal values. 
However, to our knowledge, the case of integer-valued data has not yet been comprehensively addressed.  
In this paper, we study the recovery problem for integer matrices in general. Our first fundamental question is:
\begin{problem} \label{prob:uniq}
Let ${\tt X}$ be an $N_1\times N_2$ integer matrix.
What is a minimal set of DFT coefficients required for ${\tt X}$ to be uniquely recoverable?
\end{problem}%
\noindent
Our results both characterize the minimal number of Fourier samples needed and describe their locations in the frequency domain.

Even when the sampled spectrum guarantees uniqueness, 
the reconstruction problem itself is highly nontrivial.
Standard continuous optimization methods are often ineffective over the integers, since small errors in the solution can correspond to large errors in the data~\cite{parker2014discrete,alpers2002stability}. 
Even for the simplest case of binary signals, 
many discrete optimization methods reduce to solving NP-hard problems~\cite{karp1975computational}. 
This motivates our second fundamental question:
\begin{problem} \label{prob:alg}
How can an integer matrix ${\tt X}$ be efficiently recovered from a subset of its DFT coefficients?
\end{problem}%
\noindent
To address this, we develop a dynamic programming framework that leverages approximation algorithms for efficient recovery. 
Our approach efficiently recovers integer matrices with a moderate range of values and scales to reasonably sized images. 

Beyond the theoretical interest,
these questions arise naturally in many practical settings.
Integer-valued models readily appear in  signal processing~\cite{joseph2025low,zhan2014integer}, tomography~\cite{herman2003discrete,kadu2019convex,batenburg2011dart,batenburg20093d,herman2012discrete,YAGLE2001}, 
and binary codes and image processing \cite{esedoglu2003blind,Zenzo1996,REN2002,machand-maillet2000}.
In all of these applications, data are often quantized or inherently integer-valued.

This work makes three main contributions:
We establish a minimal sampling theorem for integer matrices, identifying both the number and location of DFT coefficients required for unique recovery. 
We then develop a dynamic programming algorithmic framework for efficient reconstruction of integer matrices from partial Fourier data. 
Lastly, 
we demonstrate the advantage of a lattice-based implementation of the algorithm through a numerical comparison with integer linear programming for the practical recovery of large images.

\subsection{Related Works}
The current work directly extends the results of \cite{pei2022binary1D}, which studied the problem of recovering binary signals from sampled DFT measurements.
Their work proved that a binary signal of length $N$ requires at least $\tau(N)$  DFT coefficients for unique inversion, where $\tau(N)$ is the number of divisors of $N$.  
This result readily generalizes to integer signals.  
Their paper also shows this bound is tight when considering the set of all binary signals of length $N$.
While the authors provide an example illustrating how the problem can extend to two dimensions, no general formalization was developed. 
Our work fully develops and solves this two-dimensional generalization.

In \cite{pei2022binary1D}, a divide-and-conquer algorithm was introduced for recovering binary signals of length $N=2^\alpha$.
In our current work, we extend this framework to handle signals of arbitrary length $N$ in one dimension before further generalizing to two dimensions.
Our approach incorporates dynamic programming for efficiency, which is required by the presence of additional prime factors in $N$. 
Furthermore, while their method relied on integer linear programming, our use of lattice-based techniques
enables faster recovery and removal of the binary constraint,
allowing the new method to scale to larger signals and images. 

Our previous works~\cite{levinson2021,levinson2023} studied related partial data recovery problems for binary vectors and matrices in the bandlimited setting, which restricts the known DFT coefficients to a frequency band defined by $|k|,|l| \le L$ for a bandwidth parameter $L$.
This bandlimited sampling models blurring of a signal or image with a low-pass filter.
The restrictive shape of the band makes it challenging to identify simple formulas for bandwidths $L$ that guarantee uniqueness for general $N_1, N_2$. 
Therefore, these works were limited to the case where $N_1=N_2=p^\alpha$ for a prime $p$, or $N_1$ and $N_2$ were different primes.  
While the algorithm framework developed in the current work builds on these methods, the presence of multiple prime factors in $N_1$ and $N_2$ enables a new dynamic programming approach. 
Moreover, in the bandlimited setting, the minimal band required for uniqueness includes more coefficients than are strictly required.
In contrast, this work focuses on recovery from a minimal set of DFT coefficients. 
The pass-band shape introduces an asymmetry between the number of DFT coefficients available to each dynamic programming subproblem. 
Subproblems with more data have improved stability,
so the bandlimited setup can facilitate more stable algorithms which leverage both lattice methods and integer linear programming.

Further related works address binary reconstruction from incomplete Fourier data, both in the bandlimited setting \cite{nashold2002synthesis,nashold1989synthesis} and with unrestricted sampling \cite{mao2012reconstruction}. 
Other approaches exploit discrete-valued models in alternative bases to enhance reconstruction and guide sampling strategies \cite{vetterli2002sampling}. 
Our uniqueness results focus on the algebraic structure of the DFT, and many works similarly leverage algebraic properties to improve reconstruction methods and establish uncertainty principles \cite{stolk2010algebraic,tropp2008linear,tao_2005_1}.
Our implementations use the LLL approximation algorithm,
which has applications across a wide range of fields, 
including cryptography~\cite{Lenstra1982,Coppersmith1996},
integer relation~\cite{hastad1989,bailey2009},
and combinatorial optimization~\cite{Lenstra1983,Aardal2000}.

Similar work in the discrete tomography setting reconstructs an image from a collection of projection values,
that is, 
sums of the image along a prescribed set of directions~\cite{hajdu2001algebraic,BRUNETTI20132281,HAJDU2005}.  
By incorporating additional structural assumptions, such as convexity or connectedness, these methods often obtain strong uniqueness guarantees which require only a small number of projection directions~\cite{gardner1999tomography,lungo1999tomography}.  
Corresponding efficient reconstruction algorithms have also been developed~\cite{brunetti2000reconstruction,chrobak1999convex}.
Our work differs in that we impose no structural assumptions on the image, considering only the integer-valued signal model.

Our two-dimensional theory can likewise be interpreted in terms of directional projections.  
However, 
the computational challenge is fundamentally different. 
Rather than directly reconstructing the image from known projections, 
we begin with DFT coefficient measurements
and first recover the projection values by solving an underdetermined system.  
Once these projection values have been recovered, 
image reconstruction is immediate. 
Finally, 
this work considers periodic Discrete Radon Transform projections~\cite{grigoryan-superposition,grigoryan-book,gertner2002new, hsung1996discrete,kingston2007generalised},
whereas the closely related Mojette Transform produces non-periodic projections~\cite{guedon2005mojette,guedon1995psychovisual}.


\subsection{Outline and Notation}

We begin in \cref{sec:background} with background on the problem, including a summary of key results for one-dimensional signals. 
This section also develops additional one-dimensional framework used in the extension to two dimensions. 
\Cref{sec:theory} addresses \cref{prob:uniq}, deriving the number and location of DFT coefficients required for unique recovery of any integer matrix. 
In \cref{sec:algs}, we present the algorithm framework for recovery from a minimal set of DFT coefficients.
Finally, 
\cref{sec:numerics} contains supporting numerical simulations.
\\

Throughout the paper, we use boldface letters to denote vectors (as in ${\bf x}$), with entries written as $x_n$.
Typewriter-style letters denote matrices (as in ${\tt X}$), with entries $X_{mn}$.
We denote the $n\times n$ identity matrix by $\idmat{n}$.
The complex conjugate of a complex number $x$ is denoted by $x^*$.
We use the standard number-theoretic notation $\gcd(m,n)$ and $\lcm(m,n)$ for the greatest common divisor and least common multiple of two integers $m$ and $n$, respectively.
If an integer $m$ is a divisor of $n$, we write $m \divs n$, and otherwise $m \ndivs n$.
We also write $\tau(n)$ for the number of divisors of $n$, $\phi(n)$ for Euler's totient function (the number of positive integers less than $n$ that are relatively prime with $n$), and $\omega(n)$ for the number of distinct prime factors of $n$.   Lastly, we let $\uroot{n}{}=\e{-2\pi\cu/n}$ be an $n$th primitive root of unity.


\section{Problem Setup and Background}
\label{sec:background}

Let ${\tt X}$ be an $N_1 \times N_2$ matrix. The two-dimensional discrete Fourier transform (DFT) of ${\tt X}$ is defined by
\begin{equation}
\label{eq:DFT_2D}
\tilde{X}_{kl} = \sum_{m=0}^{N_1-1}\sum_{n=0}^{N_2-1}X_{mn} \e{-2\pi\cu \left(\frac{mk}{N_1}+\frac{nl}{N_2}\right)} \ .
\end{equation}
The Fourier coefficients $\tilde{X}_{kl}$ are periodic in each index independently, satisfying $\tilde{X}_{k,l} = \tilde{X}_{k+sN_1, l+tN_2}$ for any integers $s$ and $t$. 
A complete set of $N_1N_2$ DFT coefficients consists of taking all $\tilde{X}_{k l}$ with indices $k$ and $l$ in the ranges $s \le k <s+N_1$ and $t \le l <t+N_2$, for some choice of $s$ and $t$. 
We denote this $N_1 \times N_2$ matrix of DFT coefficients by $\tilde{\tt X}$, and typically choose $s=t=0$ so that the index ranges begin at zero.   

While the DFT maps ${\tt X}$ to $\tilde{\tt X}$, the inverse DFT recovers ${\tt X}$ from $\tilde{\tt X}$.  
This inverse transformation is defined by
\begin{equation}
\label{eq:DFT_inv_2D}
 X_{mn} = \frac{1}{N_1N_2} \sum_{k=0}^{N_1 - 1} \sum_{l = 0}^{N_2 - 1}  \tilde{X}_{kl}  \e{2\pi\cu \left(\frac{mk}{N_1}+\frac{nl}{N_2}\right)} \ ,
\end{equation}
and can equivalently be expressed using any indexing that forms a complete set of DFT coefficients.  

We are interested in the partial data question of \cref{prob:uniq}.  If only some of the DFT coefficients $\tilde{X}_{kl}$ are known, can ${\tt X}$ be uniquely reconstructed? 
In general, this is impossible as seen in \cref{eq:DFT_inv_2D}, where changing the value of only one DFT coefficient $\tilde{X}_{kl}$ will change all elements of ${\tt X}$.  
However, incorporating the {\it a priori} knowledge that ${\tt X}$ is integer-valued can make this inverse problem uniquely solvable.  
Existing results on this problem focus primarily on the one-dimensional case  for binary signals. 
Since several aspects of the one-dimensional setting carry over to the two-dimensional problem studied here, we will first summarize the relevant prior results. 
Later in this section, we derive a key property of one-dimensional signals that will be used in the two-dimensional theory.


\subsection{One-Dimensional Results}

The question of uniqueness posed in \cref{prob:uniq} has been fully answered in one dimension~\cite{pei2022binary1D}.  The one-dimensional DFT and inverse DFT are defined by 
\begin{align*} 
    \tilde{x}_k = \sum_{n=0}^{N-1} x_n \e{2\pi \cu nk/N} \qquad \qquad , 
    \qquad \qquad x_n = \frac{1}{N}\sum_{k=0}^{N-1} \tilde{x}_k \e{-2\pi \cu nk/N} \ .
\end{align*}
If ${\bf x}$ is a vector of length $N$ consisting of integer entries, then ${\bf x}$ can be uniquely determined by measurements of the $\tau(N)$ DFT coefficients,
\begin{equation} \label{eq:1D_coeffs}
    S = \set{\tilde x_d : d \divs N}.
\end{equation}
The fact that $S$ is a sufficient set of DFT coefficients for uniqueness comes from the irreducibility of cyclotomic polynomials, which is used to prove the following result.

\begin{lemma}[{\cite[Lemma 4]{pei2022binary1D}}] \label{lem:dep1D}
    Let ${\bf x}$ be an integer signal of length $N$.
    If $\gcd(k, N) = d$ and $\tilde x_k = 0$, then $\tilde x_{k'} = 0$ for all $k'$ such that $\gcd(k', N) = d$.
\end{lemma}

Consider two integer signals ${\bf y}$ and ${\bf z}$ of length $N$. 
Since the DFT is bijective, ${\bf y} = {\bf z}$ if and only if $\tilde{{\bf y}} = \tilde{{\bf z}}$.  Let ${\bf x} = {\bf y} - {\bf z}$. 
Applying \cref{lem:dep1D} to the integer signal ${\bf x} = {\bf y} - {\bf z}$ implies that $\tilde{\bf x} = {\bf 0}$ only if $\tilde{x}_k = 0$ for each $k \in S$.
Therefore, two integer signals ${\bf y}$ and ${\bf z}$ are equal only if $\tilde{y}_d = \tilde{z}_d$ for each divisor $d$ of $N$. 

It is also proved in \cite{pei2022binary1D} that every DFT coefficient in the set $S$ from \cref{eq:1D_coeffs} is required to guarantee uniqueness in general, even with further restrictions on the integer values.  
This is not obvious, as many vectors may be uniquely determinable from a smaller set of DFT coefficients. 
For example, consider the class of binary signals ${\bf x}$ where all entries are either 0 or 1.
The only binary signal satisfying $\tilde{x}_0 = 0$ is ${\bf x} = {\bf 0}$, 
so only the $k = 0$ coefficient is required to distinguish ${\bf 0}$ from all other binary signals.
However, for any $N$ and divisor $d$ of $N$, there exist binary signals ${\bf y}$ and ${\bf z}$ such that 
$ \tilde y_k \ne \tilde z_k$ if and only if $ \gcd(k, N) = d$.
Therefore, the signals ${\bf y}$ and ${\bf z}$ may only be distinguished by a DFT coefficient of frequency $k$ satisfying $\gcd(k, N) = d$, 
preventing the inverse problem from being solved uniquely if $\tilde{x}_d$ is removed from $S$.
We will require a weaker version of this statement, which only considers the divisor $d=1$. 
This relaxation allows for a simpler proof, which is provided below.  Note that while the discussion has focused on distinguishing the pair of binary vectors ${\bf y}$ and ${\bf z}$, the following lemma relates to the vector ${\bf x} = {\bf y} - {\bf z}$ which has entries in $\{-1,0,1\}$.

\begin{lemma}[{Special case of \cite[Lemma 2]{pei2022binary1D}}]
\label{lem:amb1D}
    For any positive integer $N$, there exists a signal ${\bf x}$ of length $N$ with entries in $\set{-1, 0, 1}$ such that $
        \tilde{x}_k \ne 0 $ if and only if  $\gcd(k, N) = 1$.
\end{lemma}
\begin{proof}
Let $N = \prod_{t = 1}^\omega p_t^{\alpha_t}$ be the prime factorization of $N$.
For each subset of indices $T \subseteq \set{1, \ldots, \omega}$, define $n_T = N\sum_{t \in T} \frac{1}{p_t} \bmod{N}$. 
We construct the vector ${\bf x}$ entrywise by
\begin{equation}
\label{eq:xn_def}
    x_n = \begin{cases}
        (-1)^{|T|} & \text{if } n=n_T \text{ for some } T \\
        0 & \text{else}
    \end{cases}.
\end{equation} 
We will first show that if $n_{T_1}=n_{T_2}$, then $T_1=T_2$, which proves that ${\bf x}$ is well defined.  

If $n_{T_1}=n_{T_2}$, then by definition, for some $j\in \Z$ we have 
\begin{equation}
\label{eq:nT1}
    N\sum_{t \in T_1} \frac{1}{p_t} - N\sum_{t \in T_2} \frac{1}{p_t} + jN = 0. 
\end{equation}
Suppose for contradiction that there exists a $t'\in T_1 \setminus T_2$. 
We then rewrite \cref{eq:nT1} as
\begin{equation}
\label{eq:nT2}
  \frac{N}{p_{t'}} + N\sum_{t \in T_1\setminus\{t'\}} \frac{1}{p_t} - N\sum_{t \in T_2} \frac{1}{p_t} + jN  = 0.   
\end{equation}
We observe that
$p_{t'}^{\alpha_{t'}}$ divides every term in \cref{eq:nT2} except the first term $N/p_{t'}$.   Hence the left hand side of \cref{eq:nT2} is not divisible by $p_{t'}^{\alpha_{t'}}$, contradicting that it can equal 0.  We can thus conclude that $T_1\subseteq T_2$, and by symmetry, we must have $T_1=T_2$.  Therefore ${\bf x}$ is well defined.
 
With this choice of ${\bf x}$, we can compute any DFT coefficient as
\begin{align}
\label{eq:lemma_dft}
     \tilde{x}_k
    = \sum_{n = 0}^{N - 1} x_{n} \e{-2\pi\cu nk/N} 
    = ~~~\sum_{\mathclap{T \subseteq \set{1, \ldots, \omega}}}x_{n_T} \e{-2\pi\cu n_Tk/N} 
    = ~~~\sum_{\mathclap{T \subseteq \set{1, \ldots, \omega}}} (-1)^{\abs{T}} \exp\left(-2\pi\cu k\sum_{t \in T}1/p_t\right) \ . 
\end{align}
If $\gcd(k, N) \ne 1$,
then there exists $1 \le t' \le \omega$ such that $p_{t'} \divs k$. 
Continuing from \cref{eq:lemma_dft}, 
we can further decompose the DFT summation over index sets based on membership of $t'$,
\begin{align}
    \tilde{x}_k 
    & = \sum_{T \not\ni t'} (-1)^{\abs{T}} 
    \left[\exp\left(-2\pi\cu k\sum_{t \in T}1/p_t\right) - \exp\left(-2\pi\cu k\left(\sum_{t \in T}1/p_t + 1/p_t'\right)\right)\right] \nonumber\\
    & = \sum_{T \not\ni t'} (-1)^{\abs{T}} 
    \left[\exp\left(-2\pi\cu k\sum_{t \in T}1/p_t\right) - \exp\left(-2\pi\cu k\sum_{t \in T}1/p_t\right)\exp\left(-2\pi \cu k/p_t'\right)\right] \ .
    \label{eq:sum_decomp}
\end{align}
This latter expression sums over all subsets $T$ of indices that do not contain index $t'$,
adding both the term corresponding to that subset $T$ and the term corresponding to the subset $T\cup\{t'\}$. 
As we have assumed that $p_{t'}\divs k$,
we have that $k/p_{t'}$ is an integer, making $\e{-2\pi\cu k/p_{t'}} = 1$. 
Therefore, the expression in \cref{eq:sum_decomp} reduces to
\begin{equation*}
    \tilde{x}_k = \sum_{T' \not\ni t'} (-1)^{\abs{T'}} 
    \left[\exp\left(-2\pi\cu k\sum_{t \in T'}1/p_t\right) - \exp\left(-2\pi\cu k\sum_{t \in T'}1/p_t\right)\right] = 0 \ ,
\end{equation*}
as desired.
For the reverse direction,
since ${\bf x}$ is nonzero, there exists $k$ such that $\tilde x_k \ne 0$.
Our previous work implies that $\gcd(k, N)$ must be 1.  Thus, \cref{lem:dep1D} implies that $\tilde x_{k'} \ne 0$ for each $k'$ such that $\gcd(k', N) = 1$.
\end{proof}
To clarify the construction outlined in the proof of \cref{lem:amb1D}, consider the following illustrative example.
For convenience, we adopt the compact notation $\uroot{N}{}$ for the $N$th primitive root of unity, $\uroot{N}{} = \e{-2\pi\cu/N}$.
This notation will be used in the rest of this work.
\begin{example} \label{ex:amb1D_6}
    Let $N = 6$. The prime factorization of 6 is $6 = p_1^1 \cdot p_2^1$, where $p_1 = 2$ and $p_2 = 3$.
    The power set of $\set{1, 2}$ is 
    \begin{equation*}
        \mathcal{P}(\set{1, 2}) = \set{\eset, \set{1}, \set{2}, \set{1, 2}} \ .
    \end{equation*}
    Computing all possible value of $n_T=6\sum_{i \in T} \frac{1}{p_i}$ yields,
    \begin{align*}
        n_\eset &= 0 \ ,
        &&&
        n_{\set{1}} &= 6 \cdot \frac{1}{p_1} = 3 \ ,
        \\
        n_{\set{2}} &= 6 \cdot \frac{1}{p_2} = 2 \ ,
        &&&
        n_{\set{1, 2}} &= 6 \cdot \frac{1}{p_1} + 6 \cdot \frac{1}{p_2} = 5 \ .
    \end{align*}
    This corresponds to the signal
    ${\bf x} = 
    \begin{bmatrix}
        1 & 0 & -1 & -1 & 0 & 1
    \end{bmatrix}$, 
    where the appropriate value of $\pm1$ is placed at the calculated indices.   
    Computing its DFT yields,
    \begin{equation*}
        \tilde{\bf x} = 
        \begin{bmatrix}
            0 & 2 - 2\uroot{3}{} & 0 & 0 & 0 & 2 - 2\uroot{3}{*}
        \end{bmatrix} \ ,
    \end{equation*}
    where we have made the simplification $\uroot{6}{2} = \uroot{3}{}$.
    Note that ${\bf x}$ exactly satisfies $\tilde{x}_k\neq 0$ for only $k=1$ and $k=5$, as desired.
\end{example}

Our current work generalizes the results of \cref{lem:dep1D,lem:amb1D} to two-dimensional integer signals of any size.
Limited previous results pertaining to similar problems are available for two dimensions.
For example, the results of \cite{levinson2023,pei2023binary} show that if $N_1 = N_2$ is prime, then an integer signal ${\tt X}$ may be recovered from a minimal set of $N + 2$ DFT coefficients.
This result will follow as a special case of our \cref{thm:dep2D,thm:amb2D}. 

Having established the main one-dimensional results, we conclude this section with a key property of the DFT of frequency domain decimated signals. This result is independent of the integer constraints and serves as the main tool for 
developing the efficient inversion algorithms presented in \cref{sec:algs}.
We begin by introducing the relevant definition.


\NiceMatrixOptions{xdots/horizontal-labels}
 \tikzset
   {
     nicematrix/brace/.style =
       {
         decoration =
           {
             calligraphic brace ,
             amplitude = 0.4 em ,
             raise = -0.15 em
},
line width = 0.1 em , decorate ,
} }

\begin{definition} \label{def:dec}
    Let ${\bf x}$ be a signal of length $N$.
    If $d \divs N$, then the frequency decimated signal ${\bf x}^{(N/d)}$ of length $N/d$ is defined entrywise by
    \begin{equation*}
        x^{(N/d)}_m = \sum_{n = 0}^{d - 1}x_{m + nN/d},
        \qquad \text{for}~0\le m < \frac{N}{d}.
    \end{equation*}
\end{definition}
\noindent
The process of constructing ${\bf x}^{(N/d)}$ from ${\bf x}$ corresponds to decimating by a factor of $d$ in the frequency domain. 
The following brief lemma justifies this terminology, 
as the decimated signal retains every $d$th DFT coefficient of ${\bf x}$.  
\begin{lemma} \label{lem:dec_in_freq}
    Let ${\bf x}$ be a signal of length $N$.
    If $d$ divides $N$, then for $k=k'd$,  $\tilde{x}_k = \tilde{x}_{k'}^{(N/d)}$.
\end{lemma}
\begin{proof}
Let $k = dk'$. 
Then the $k$th DFT coefficient can be written as 
\begin{equation*}
    \tilde{x}_{dk'} 
    = \sum_{n = 0}^{N - 1} x_n \uroot{N}{nk'd} 
    = \sum_{n = 0}^{N - 1} x_n \uroot{N/d}{nk'} 
    = \sum_{m = 0}^{N/d - 1} 
    \bigg[ \sumstack{n}{n=m}{N/d}{x_n} \bigg] \uroot{N/d}{mk'} 
    = \sum_{m = 0}^{N/d - 1} x_m^{(N/d)} \uroot{N/d}{mk'} ,
\end{equation*}
which shows that $\tilde{x}_{k} = \tilde{x}^{(N/d)}_{k'}$. 
\end{proof}
The signal ${\bf x}^{(N/d)}$ may be alternatively defined by the matrix equation,
\begin{equation}
\label{eq:matrix_form}
    {\bf x}^{(N/d)} = 
    \begin{bmatrix}
        \idmat{N/d} & \cdots & \idmat{N/d}
    \end{bmatrix} 
    {\bf x} \ , 
\end{equation}
where the $(N/d) \times N$ matrix consists of $d$ blocks of the $(N/d) \times (N/d)$ identity matrix.
\Cref{lem:dec_in_freq} implies that if $\begin{bmatrix}
    \idmat{N/d} & \cdots & \idmat{N/d}
\end{bmatrix} 
{\bf y} = {\bf x}^{(N/d)}$ for some other signal $\mathbf{y}$, then it must hold that $\tilde{y}_d=\tilde{x}_d$.
This consequence of \cref{lem:dec_in_freq} will be used later to optimize the reconstruction algorithms.

\section{Uniqueness Results}
\label{sec:theory}

We now develop the tools needed to address the uniqueness question of \cref{prob:uniq}.
We begin by introducing an equivalence relation on the DFT coefficients of integer matrices.
Sampling at least one DFT coefficient from each equivalence class will be sufficient to ensure the uniqueness of recovery. 
Next, we show that this sampling strategy is optimal in the sense that if any equivalence class is not sampled, then there exist distinct integer matrices that cannot be distinguished apart from each other. 
We conclude with results that explicitly count the minimum number of DFT coefficients required for uniqueness.

We note that constructions related to our \cref{def:subsig} appear in several works, including as splitting signals \cite{grigoryan-superposition,grigoryan-book} and discrete variants of the Radon transform \cite{gertner2002new, hsung1996discrete,kingston2007generalised}.  This definition plays a key role later on in breaking down the two-dimensional reconstruction problem into a carefully chosen sequence of one-dimensional problems.  Although this general idea is used in these other works, none address the underdetermined reconstruction problem considered in this paper.

\subsection{DFT Coefficient Equivalence}

Let ${\tt X}$ be an $N_1 \times N_2$ matrix and let $N = \lcm(N_1, N_2)$.
For every $0 \le k < N_1$ and $0 \le l < N_2$, finding the least common denominator yields,
\begin{equation*}
    \frac{mk}{N_1} + \frac{nl}{N_2} = \frac{mkN_2' + nlN_1'}{N}  \ ,
\end{equation*}
where $N_1' = N_1/\gcd(N_1,N_2)$ and $N_2' = N_2/\gcd(N_1,N_2)$.  This directly implies that 
$\uroot{N_1}{mk}\uroot{N_2}{nl} = \uroot{N}{mkN_2'+nkN_1'}$.
Thus, we may express the DFT coefficient $\tilde{X}_{kl}$ as a one-dimensional DFT, 
\begin{equation*} 
    \tilde{X}_{kl} 
    = \sum_{m = 0}^{N_1 - 1}\sum_{n = 0}^{N_2 - 1} {X}_{mn} \uroot{N_1}{mk}\uroot{N_2}{nl}
    = \sum_{m = 0}^{N_1 - 1}\sum_{n = 0}^{N_2 - 1} {X}_{mn} \uroot{N}{mkN_2' + nlN_1'}
    = \sum_{j = 0}^{N - 1} \bigg [\sumstack{m, n}{mkN_2' + nlN_1' = j}{N}{X_{mn}} \bigg] \uroot{N}{j} \ .
\end{equation*}
Additionally, 
for every $0 \le \lambda < N$, 
the DFT coefficient $\tilde{X}_{\lambda k,\lambda l}$ takes the form
\begin{equation}
\label{eq:recto1D_lambda}
    \tilde{X}_{\lambda k, \lambda l} = 
    \sum_{j = 0}^{N - 1} y_j^{(k,l)}\uroot{N}{\lambda j} \ ,
    \qquad\qquad
     y_j^{(k,l)} = 
    \sumstack{m, n}{mkN_2' + nlN_1' = j}{N}{X_{mn}} ,
    \qquad 
    \text{for}~0 \le j < N \ ,
\end{equation}
so
$\tilde{X}_{\lambda k, \lambda l}$ is the $\lambda$th DFT coefficient of the signal 
$\mathbf{y}^{(k, l)} = \begin{bmatrix}
    y_0^{(k, l)} & \cdots & y_{N - 1}^{(k, l)}
\end{bmatrix}$. 
Recall that the length of this signal is $N=\lcm(N_1,N_2)$.  
However, 
in general, 
the sum in \cref{eq:recto1D_lambda} may be empty for some choices of $j$ as there may be no $m$ and $n$ that satisfy the congruence condition. 
We will reduce the length of the vector ${\bf y}$ by discarding the empty sums, 
which carry no information about ${\tt X}$ since they are zero for any matrix.

First, 
note that when $\gcd(k, N_1) \ne 1$,
$\uroot{N_1}{k}$ is not a primitive $N_1$th root of unity.
To account for this, define $D_1$, $D_2$, and $D$ by
\begin{equation}
\label{eq:D_notation}
  D_1 = N_1/\gcd(k, N_1) \qquad , \qquad  D_2 = N_2/\gcd(l, N_2) \qquad , \qquad D = \lcm(D_1, D_2) \ . 
\end{equation}
By construction, $\uroot{N_1}{k}$ and $\uroot{N_2}{l}$ are primitive $D_1$th and $D_2$th roots of unity,
respectively.
Thus,
every root of unity $\uroot{N_1}{mk}\uroot{N_2}{nl}$ appearing in the DFT equation for $\tilde{X}_{kl}$ can be expressed as $D$th root of unity.
Therefore, if the congruence $mkN_2'+nlN_1'=j \pmod N$ has a solution, then $\uroot{N}{j}$ must be a $D$th root of unity.
As $\uroot{N}{j}$ is a $D$th root of unity only if $d = N/D$ satisfies $d\divs j$,
the signal ${\bf x}^{(k, l)}$ defined by $x^{(k, l)}_j = y^{(k, l)}_{jd}$ for $0 \le j < D$ contains every entry of ${\bf y}^{(k, l)}$ that is not identically zero.
We can conclude that 
\begin{equation}
\label{eq:X_lambdak}
    \tilde{X}_{\lambda k,\lambda l} = \tilde{x}_\lambda^{(k, l)} ,
\end{equation}
since the summation in \cref{eq:recto1D_lambda} may be restricted to the $D$ indices $j$ that are multiples of $d$.
This structure motivates the following definition, 
which allows us to apply the techniques for one-dimensional integer signals to integer matrices.

\begin{definition}[Subsignal] \label{def:subsig}
    Let ${\tt X}$ be an $N_1 \times N_2$ integer matrix. 
    Let $N_1' = N_1/\gcd(N_1, N_2)$, $N_2' = N_2/\gcd(N_1, N_2)$ and $N = \lcm(N_1, N_2)$.
    For any $k$ and $l$, we define the $(k, l)$-subsignal ${\bf x}^{(k,l)}$ of ${\tt X}$ entrywise by,
    \begin{equation*} 
        x_{j}^{(k, l)} = 
        \sumstack{m,n}{mkN_2' + nlN_1' = jd}{N}{X_{mn}\ },
        \qquad
        \text{for}~0 \le j < D \ ,
    \end{equation*}
    where $D$ is defined as in \cref{eq:D_notation} and $d = N/D$.
\end{definition}

As the $(k, l)$-subsignal is a one-dimensional integer signal of length $D$, \cref{lem:dep1D} states that if $\tilde x^{(k,l)}_\lambda = 0$, then $\tilde x^{(k,l)}_{\lambda'} = 0$ for all $\lambda'$ such that $\gcd(\lambda', D) = \gcd(\lambda,D)$. 
Therefore, if $\tilde{x}^{(k,l)}_1=\tilde{X}_{kl}$ is zero, then for all $\lambda$ such that $\gcd(\lambda,D)=1$, we have $\tilde{x}^{(k,l)}_\lambda=\tilde{X}_{\lambda k,\lambda l}=0$. 
This implies that if we sample $\tilde{X}_{kl}$, then there is no additional information provided by any $\tilde{X}_{\lambda k,\lambda l}$ with $\gcd(\lambda,D)=1$,
which proves the ensuing \cref{thm:dep2D}. 
Before stating the theorem, we first express the relationship between dependent coefficients as an equivalence relation to simplify its formulation.

\begin{definition} \label{def:dep2D}
    Let ${\tt X}$ be an $N_1 \times N_2$ integer matrix.
    Two DFT coefficients $\tilde{X}_{k l}$ and $\tilde{X}_{k' l'}$ are equivalent, denoted by $(k, l) \sim (k', l')$, if there exists a nonzero integer $\lambda$ that is relatively prime with $D$ satisfying 
    \begin{equation*}
        k' = \lambda k \pmod{N_1} \qquad \text{and} \qquad l' = \lambda l \pmod{N_2} \\ .
    \end{equation*}
\end{definition}

To verify that the equivalence relation $\sim$ is symmetric, fix $k$ and define $D_1$ as in \cref{eq:D_notation}.
Suppose $k' = \lambda k \pmod{N_1}$ with $\gcd(\lambda, D) = 1$, as in \cref{def:dep2D}. 
As $\lambda$ is relatively prime with $D$, $\lambda^{-1}$ exists $\pmod{D}$.
Since $D_1 \divs D$, $\lambda^{-1}$ is also the multiplicative inverse of $\lambda$ modulo $D_1$. 
Let $d_1 = \gcd(k, N_1)$.
Note that we must have $d_1 \divs k'$, which justifies that 
\begin{align*}
    k' = \lambda k \pmod{N_1} 
    & \iff (k'/d_1) = \lambda (k/d_1 )\pmod{N_1/d_1} \\
    & \iff (k'/d_1) = \lambda (k/d_1 ) \pmod{D_1} \\
    & \iff \lambda^{-1}(k'/d_1) = (k/d_1 ) \pmod{D_1} \\
    & \iff \lambda^{-1}k' = k \pmod{N_1} \ .
\end{align*}
Performing an identical computation for $l,l'$ shows that $l = \lambda^{-1}l' \pmod{N_2}$, so $(k', l') \sim (k, l)$.

We can now state the main result of this section.
\begin{theorem} \label{thm:dep2D}
    Let ${\tt X}$ be an $N_1 \times N_2$ integer matrix.
    For every $0 \le k < N_1, 0 \le l < N_2$, and every frequency $(k', l') \sim (k, l)$,
    $\tilde{\tt X}_{kl} = 0$ if and only if $ \tilde{\tt X}_{k'l'} = 0$.
\end{theorem}

\Cref{thm:dep2D} provides an upper bound for \cref{prob:uniq}.
It implies that a single representative from each equivalence class in $\tilde{\tt X}/\sim$ forms a sufficient set of given DFT coefficients for unique recovery of an integer matrix ${\tt X}$.
Therefore, the number of DFT coefficients required for unique recovery is at most the number of elements of $\tilde{\tt X}/\sim$.

The next section will focus on proving that this upper bound provided by \cref{thm:dep2D} is tight.
To do so, for any given frequency $(k, l)$, we produce a pair of integer $N_1 \times N_2$ matrices ${\tt X}$ and ${\tt Y}$ such that $\tilde{\tt X}$ and $\tilde{\tt Y}$ differ only at frequencies $(k', l') \sim (k, l)$, and are equal at all other frequencies.
This shows that a representative of the equivalence class $(k, l)$ is required to uniquely reconstruct an integer matrix.

\subsection{Minimal DFT Sampling} 
Our next result shows that the equation which determines the $j$th entry of the $(k, l)$-subsignal,
\begin{equation}
    \label{eq:to_satisfy}
    mkN_2' + nlN_1' = jd \pmod{N} \ ,
\end{equation}
evenly partitions the entries of ${\tt X}$.  The proof of \cref{lem:subsig} relies on the underlying algebraic structure of the partition in \cref{eq:to_satisfy}. For readability, we defer the proof to \apref{ap:subsig}.

\begin{lemma} \label{lem:subsig}
    Let ${\tt X}$ be an $N_1 \times N_2$ integer matrix.
    For any frequency $(k, l)$, each entry $x_{j}^{(k, l)}$ of the $(k, l)$-subsignal is a sum of $M = N_1N_2/D$ entries of ${\tt X}$,
    where $D$ is defined as in \cref{eq:D_notation}.
\end{lemma}
The next lemma uses \cref{lem:subsig} to generalize \cref{lem:amb1D} to the two-dimensional setting. 
For any fixed frequency $(k, l)$,
we can construct a matrix ${\tt X}$ with entries in $\set{-1, 0, 1}$ such that only the DFT coefficients at frequencies equivalent to $(k, l)$ are potentially nonzero.
This is accomplished by making the $(k, l)$ subsignal of ${\tt X}$ proportional to the signal in \cref{lem:amb1D}.
Moreover, 
every DFT frequency $(k', l')$ that does not appear in the $(k, l)$-subsignal must be 0.

\begin{lemma} \label{lem:amb2D}
    Let $N_1$ and $N_2$ be any positive integers.
    Fix $(k, l)$ and define $D$ as in \cref{eq:D_notation}.
    Let ${\bf x}$ be the signal of length $D$ defined in \cref{eq:xn_def} from \cref{lem:amb1D}.
    For the matrix ${\tt X}$ defined by \begin{equation}
        \label{eq:lemma3.5}
        X_{mn} = x_j,
        \qquad \text{for}\quad
        mkN_2' + nlN_1' = jd \pmod{N}\ ,
    \end{equation}
    the DFT coefficients satisfy
    $\tilde{X}_{k', l'} \ne 0 $ if and only if $(k', l') \sim (k, l)$.
\end{lemma}

Note that \cref{eq:lemma3.5} defines all entries of ${\tt X}$. 
This is true as for any $m$ and $n$, there exists a $j$ such that $mkN_2' + nlN_1' = jd \pmod{N}$, as demonstrated in the proof of \cref{lem:subsig}.

\begin{proof}
For the constructed matrix ${\tt X}$ in \cref{eq:lemma3.5}, the $(k,l)$-subsignal takes the form,
\begin{equation} \label{eq:subsig_prop}
    x_j^{(k, l)} 
    = \sumstack{m,n}{mkN_2'+nlN_1'=jd}{N}{X_{mn}} 
    = \sumstack{m,n}{mkN_2'+nlN_1'=jd}{N}{x_j}
    = Mx_j \ ,
\end{equation}
where the last equality uses \cref{lem:subsig} to count the $M = N_1N_2/D$ elements in the sum.
Therefore,
as the $(k,l)$-subsignal satisfies $\tilde{\bf x}^{(k, l)} = M\tilde{\bf x}$,
\cref{lem:amb1D} implies
\begin{equation}
\label{eq:2.2wD}
    \tilde{x}_\lambda^{(k,l)}=M\tilde{x}_\lambda\neq 0 \text{ if and only if }\gcd(\lambda,D)=1.
\end{equation}
Now consider any frequency $(k',l')\sim(k,l)$ of ${\tt X}$,
meaning there exists a $\lambda$ such that $k' = \lambda k \pmod{N_1}$ and $l' = \lambda l \pmod{N_2}$ with $\gcd(\lambda,D)=1$. 
\Cref{eq:X_lambdak} states that $\tilde{X}_{\lambda k, \lambda l} = \tilde{x}_\lambda^{(k, l)}$,
so the equivalence in \cref{eq:subsig_prop} yields
$\tilde{X}_{k'l'}=\tilde{X}_{\lambda k, \lambda l}=M\tilde{x}_\lambda$. 
Therefore, 
as $\gcd(\lambda,D)=1$, 
\cref{eq:2.2wD} implies that $\tilde{X}_{k'l'}\neq 0$.
Hence, 
we have $\tilde X_{k'l'}\neq0$ whenever $(k',l')\sim(k,l)$.

Next, we handle the other case,
where we must show that $\tilde{X}_{k'l'}=0$ for the remaining frequencies $(k', l')$.
Consider the Frobenius entrywise 2-norm of ${\tt X}$, which can be computed as
\begin{equation*}
    \norm{{\tt X}}_F^2 = \sum_{m=0}^{N_1-1}\sum_{n=0}^{N_2-1}|X_{mn}|^2 
    = \sum_{j = 0}^{D - 1} 
    \sumstack{m,n}{mkN_2' + nlN_1' = jd}{N}{\abs{X_{mn}}^2}
    = \sum_{j = 0}^{D - 1} M\abs{x_j}^2 
    = M\norm{{\bf x}}_2^2 \ ,
\end{equation*}
where we have again used the definition of ${\tt X}$ to replace within the sum the $M$ entries of ${\tt X}$ that equal $x_j$.
Therefore, applying the discrete Parseval relation yields,
\begin{equation} \label{eq:norms}
    \norm{\tilde{\tt X}}_F^2
    = N_1N_2 \norm{{\tt X}}_F^2
    = N_1N_2M\norm{{\bf x}}_2^2
    = \frac{N_1N_2M}{D}\norm{\tilde{\bf x}}_2^2=M^2\norm{\tilde{\bf x}}_2^2\ .
\end{equation}
As $\tilde{X}_{\lambda k,\lambda l}=M\tilde{x}_\lambda$, we also have that   
\begin{equation*}
    M^2\norm{\tilde{\bf x}}_2^2  = M^2 \sum_{\lambda = 0}^{D - 1} \abs{\tilde x_\lambda}^2 =\sum_{\lambda=0}^{D - 1}\abs{\tilde{X}_{\lambda k, \lambda l}}^2\ ,
\end{equation*}
which in combination with \cref{eq:norms} shows that
\begin{equation*}
     \norm{\tilde{\tt X}}_F^2 = \sum_{\lambda=0}^{D - 1}\abs{\tilde{X}_{\lambda k, \lambda l}}^2 \ .
\end{equation*}
This implies that all nonzero DFT coefficients are contained in the set $\set{\tilde{X}_{\lambda k,\lambda l}: 0\le \lambda < D}$. 
Therefore, we have $\tilde{X}_{k'l'}=0$ if there is no $\lambda$ such that $k'=\lambda k \pmod{N_1}$ and $l'=\lambda l \pmod{N_2}$, as desired. 
\end{proof}

We now provide two examples that illustrate the application of \cref{lem:amb2D}. 
Both examples build on \exref{ex:amb1D_6}.  
The first example demonstrates a straightforward case where $D=N$. 
The second example exhibits a more involved setup with $D \ne N$.

\begin{example} \label{ex:amb2D_2x3}
    Let $N_1 = 2, N_2 = 3$ and $k = l = 1$.
    Note that $N = \lcm(N_1, N_2) = 6$, and as $N_1$, $N_2$, $k$, and $l$ are all relatively prime, $D=N$.   Therefore, by \exref{ex:amb1D_6}, the relevant signal from \cref{lem:amb1D} is ${\bf x} = \begin{bmatrix}
        1 & 0 & -1 & -1 & 0 & 1
    \end{bmatrix}$.
    We define the following indexing matrix ${\tt S}$ by $S_{m n} = (kmN_2' + lnN_1')/d \bmod D$. 
    As $d=N/D=1$, we have $S_{mn}=3m+2n \bmod{6}$,
    which gives,
    \begin{equation*}
        {\tt S} = 
        \begin{bmatrix}
            0 & 2 & 4 \\
            3 & 5 & 1
        \end{bmatrix} \ .
    \end{equation*}
    The matrix ${\tt S}$ acts as an indexing tool in $j$ to construct ${\tt X}$ from ${\bf x}$.
    Since $S_{0 0} = 0$, 
    we set $X_{00} = x_0$.
    Similarly, $S_{01} = 2$, so $X_{01} = x_2$.
    Continuing to set $X_{m n} = x_{S_{mn}}$ produces the matrix and its DFT,
    \begin{equation*}
        {\tt X} = 
        \begin{bNiceArray}{rrr}
            1 & {-1} & 0 \\
            {-1} & 1 & 0
        \end{bNiceArray}
      \qquad , \qquad
        \tilde{\tt X} = 
        \begin{bmatrix}
            0 & 0 & 0 \\
            0 & 2 - 2\uroot{3}{} & 2 - 2\uroot{3}{*}
        \end{bmatrix} 
        \ . 
    \end{equation*}
    One can see that $\tilde{\tt X}$ only has nonzero DFT coefficients at frequencies $(1,1)$ and $(1,2)\sim (1,1)$.
\end{example}

\begin{example}
    Let $N_1 = 4, N_2 = 6$ (so $N_1' = 2, N_2' = 3$, and $N = 12$) and let $k = l = 2$.
    We compute $d_1 = \gcd(k, N_1) = 2$ and $d_2 = \gcd(l, N_2) = 2$, 
    so $D_1 = N_1/d_1 = 2$, $D_2 = N_2/d_2 = 3$, $D = \lcm(D_1, D_2) = 6$, and $d = N/D = \lcm(N_1,N_2)/6 = 2$.
    Since $D = 6$, the signal  from \exref{ex:amb1D_6}, namely
    ${\bf x} = 
    \begin{bmatrix}
        1 & 0 & -1 & -1 & 0 & 1
    \end{bmatrix}$, 
    is used again.
    The indexing matrix $S_{mn} = (kmN_2' + lnN_1')/d \bmod{N} = 3m + 2n \bmod6$ is given by
    \begin{equation*}
        {\tt S} = 
        \begin{bmatrix}
            0 & 2 & 4 & 0 & 2 & 4 \\
            3 & 5 & 1 & 3 & 5 & 1 \\
            0 & 2 & 4 & 0 & 2 & 4 \\
            3 & 5 & 1 & 3 & 5 & 1 
        \end{bmatrix} \ .
    \end{equation*}
    This gives a signal with nonzero DFT coefficients only at frequencies $(2,2)$ and $(2,4)\sim(2,2)$,
    \begin{equation*}
        {\tt X} = 
            \begin{bNiceArray}{rrrrrr}
                1 & -1 & 0 & 1 & -1 & 0 \\
                -1 & 1 & 0 & -1 & 1 & 0 \\
                1 & -1 & 0 & 1 & -1 & 0 \\
                -1 & 1 & 0 & -1 & 1 & 0 \\
            \end{bNiceArray}
      \qquad , \qquad
        \tilde{\tt X} = 
        \begin{bmatrix}
            0 & 0 & 0 & 0 & 0 & 0 \\
            0 & 0 & 0 & 0 & 0 & 0 \\
            0 & 0 & 8 - 8\uroot{3}{} & 0 & 8 - 8\uroot{3}{*} & 0 \\
            0 & 0 & 0 & 0 & 0 & 0 \\
        \end{bmatrix} \ .
    \end{equation*}
\end{example}

The following theorem applies the construction in \cref{lem:amb2D} to prove that a representative of each coefficient class is required to guarantee recovery of any $N_1 \times N_2$ integer matrix.
It further demonstrates the necessity of this requirement when restricting to  $N_1\times N_2$ binary matrices.
Combining this result with the upper bound in \cref{thm:amb2D} establishes that our equivalence relation defines the minimal sampling for unique recovery of integer matrices. 

\begin{theorem}[Minimal Sampling]
\label{thm:amb2D}
    For a fixed matrix size $N_1 \times N_2$,
    and any frequencies $0 \le k < N_1$ and $0 \le l < N_2$,
    the following hold.
    \begin{enumerate}
        \item For any integer matrix ${\tt X}$,
        there exists an integer matrix ${\tt Y}$ such that $\tilde{X}_{k'l'} \ne \tilde{Y}_{k'l'}$ if and only if $(k', l') \sim (k, l)$.
        \item There exist binary matrices ${\tt X}^{(1)}$ and ${\tt X}^{(2)}$ such that $\tilde{X}^{(1)}_{k'l'} \ne \tilde{X}^{(2)}_{k'l'}$ if and only if $(k', l') \sim (k, l)$.
    \end{enumerate}
\end{theorem}

\begin{proof}
    Let ${\tt Z}$ be the matrix from \cref{lem:amb2D} for the frequency $(k, l)$.
    For statement (1),
    define ${\tt Y} = {\tt X} + {\tt Z}$.
    Since $\tilde{Z}_{k'l'} \ne 0$ if and only if $(k', l') \sim (k, l)$,
    $\tilde{Y}_{k'l'} \ne \tilde{X}_{k'l'}$ if and only if $(k', l') \sim (k, l)$.

    For statement (2),
    define the matrices ${\tt X}^{(1)}$ and ${\tt X}^{(2)}$ by
    \begin{equation*}
        { X}^{(1)}_{mn} = \begin{cases}
            1 & \text{ if }{Z}_{mn} = 1 \\
            0 & \text{ if }{Z}_{mn} \ne 1
        \end{cases}
      \qquad , \qquad
        { X}^{(2)}_{mn} = \begin{cases}
            1 & \text{ if }{Z}_{m n} = -1 \\
            0 & \text{ if }{Z}_{m n} \ne -1
        \end{cases} \ .
    \end{equation*}
    Since the construction in \cref{lem:amb2D} has entries in $\set{-1, 0, 1}$,
    these matrices satisfy ${\tt Z} = {\tt X}^{(1)} - {\tt X}^{(2)} $. 
    In particular, 
    $\tilde{ X}^{(1)}_{k'l'} \ne
    \tilde{ X}^{(2)}_{k'l'}$
    if and only if 
    $(k', l') \sim (k, l)$,
    as desired.
\end{proof}

Note that the weaker statement (2) for binary matrices is necessary.
Considering the signal ${\tt X}$ of all zeros,
which is the only binary matrix to satisfy $\tilde{X}_{00} = 0$,
demonstrates that statement (1) is false with the addition of the binary constraint. 
To demonstrate the construction in part (2) of \cref{thm:amb2D}, 
we build upon \exref{ex:amb2D_2x3} to  generate a pair of binary matrices ${\tt X}^{(1)}$ and ${\tt X}^{(2)}$ that are indistinguishable except by a representative of the given equivalence class. 

\begin{example}
    Let $N_1 = 2, N_2 = 3$ and $k = l = 1$.
    From \exref{ex:amb2D_2x3}, we have the matrix
    \begin{equation*}
        {\tt X} = 
        \begin{bNiceArray}{rrr}
            1 & -1 & 0 \\
            -1 & 1 & 0
        \end{bNiceArray}
      \qquad , \qquad
        \tilde{\tt X} = 
        \begin{bmatrix}
            0 & 0 & 0 \\
            0 & 2 - 2\uroot{3}{} & 2 - 2\uroot{3}{*}
        \end{bmatrix} , 
    \end{equation*}
    The construction in \cref{thm:amb2D} thus yields,
    \begin{align*}
        {\tt X}^{(1)} &= 
        \begin{bmatrix}
            1 & 0 & 0 \\
            0 & 1 & 0
        \end{bmatrix}
      \qquad , \qquad
        \tilde{\tt X}^{(1)} = 
        \begin{bmatrix}
            2 & \uroot{6}{} & \uroot{6}{*} \\
            0 & 1 - \uroot{3}{} & 1 - \uroot{3}{*}
        \end{bmatrix} \\
        {\tt X}^{(2)} &= 
        \begin{bmatrix}
            0 & 1 & 0 \\
            1 & 0 & 0
        \end{bmatrix}
      \qquad , \qquad
        \tilde{\tt X}^{(2)} = 
        \begin{bmatrix}
            2 & \uroot{6}{} & \uroot{6}{*} \\
            0 & \uroot{3}{} - 1 & \uroot{3}{*} - 1
        \end{bmatrix} \ ,
    \end{align*}
    where these two matrices only disagree at DFT  frequencies $(1,1)$ and $(1,2)\sim(1,1)$.
\end{example}


\subsection{Enumeration and Algebraic Structure of Coefficient Classes}

Our work in the previous section demonstrated that \cref{thm:dep2D} gives a tight upper bound on the number of DFT measurements required for uniqueness.
This section investigates the algebraic structure of the equivalence classes,
which may be used to count the number of elements of $\tilde{\tt X}/\sim$.

The DFT coefficient equivalence classes of an $N_1 \times N_2$ integer matrix may be identified with cyclic subgroups of $\mathbb{Z}_{N_1} \times \mathbb{Z}_{N_2}$. 
First, if $(k, l)$ and $(k', l')$ satisfy $(k, l) \sim (k', l')$, then there exists $0 < \lambda < D$ such that $k' = \lambda k \pmod{N_1}$, $l' = \lambda l\pmod{N_2}$. 
This is equivalent to saying $(k', l') = \lambda(k, l)$  in $\mathbb{Z}_{N_1} \times \mathbb{Z}_{N_2}$, 
which shows that $\langle (k', l') \rangle \subseteq \langle (k, l) \rangle$ in $\mathbb{Z}_{N_1} \times \mathbb{Z}_{N_2}$. 
The symmetry of the equivalence relation $\sim$ implies 
$\langle (k, l) \rangle = \langle (k', l') \rangle$.

For the reverse direction, now suppose that 
$\langle (k, l) \rangle = \langle (k', l') \rangle$. The order of $\langle (k,l) \rangle$ in $\mathbb{Z}_{N_1} \times \mathbb{Z}_{N_2}$ is known to be 
\begin{equation*}
    | \langle (k,l) \rangle | = \lcm\left (\frac{N_1}{\gcd(k, N_1)}, \frac{N_2}{\gcd(l, N_2)} \right) = D \ .
\end{equation*}
Since $(k', l') \in \langle (k, l) \rangle$, 
and $D$ is the order of $(k, l)$ and $(k', l')$,
there exists $0 \le \lambda < D$ such that 
$(k', l') = \lambda(k, l)$.
Properties of cyclic groups thus imply that the order of $(k', l')$ in 
$\mathbb{Z}_{N_1} \times \mathbb{Z}_{N_2}$
is $D / \gcd(\lambda, D)$. 
As the order of $(k', l')$ is $D$, we must have $\gcd(\lambda, D) = 1$ which shows that $(k, l) \sim (k', l')$.

This correspondence yields an approach for computing the theoretical number of DFT coefficients required for inversion,
which is stated as the following corollary to \cref{thm:dep2D,thm:amb2D}.
\begin{corollary} \label{cor:coeff_count}
    The minimal number of DFT coefficients required to uniquely recover an integer $N_1 \times N_2$ matrix is the number of cyclic subgroups of $\mathbb{Z}_{N_1} \times \mathbb{Z}_{N_2}$.
\end{corollary}
\noindent
Generally, the number of cyclic subgroups of $\mathbb{Z}_{N_1} \times \mathbb{Z}_{N_2}$ may be computed from
\begin{equation} \label{eq:n_subgroups}
    c(N_1, N_2) = \sum_{a\divs N_1, b\divs N_2} \phi(\gcd(a, b)),
\end{equation}
where the sum ranges over pairs of divisors of $N_1$ and $N_2$, and $\phi$ is the totient function~\cite{hampejs2014representingcountingsubgroupsgroup}.
The following examples show a few cases where this expression can be simplified.

\begin{example}[Coprime Dimensions]
    
If $\gcd(N_1, N_2) = 1$, then 
\begin{equation*}
    \mathbb{Z}_{N_1} \times \mathbb{Z}_{N_2} \cong \mathbb{Z}_{N_1N_2},
\end{equation*}
which is cyclic.
In any cyclic group, there is exactly one subgroup of order $d$ for each divisor $d$ of its order, so we have $\tau(N_1N_2) = \tau(N_1)\tau(N_2)$ coefficient classes.

\end{example}

\begin{example}[Prime Power Square]
When $N_1 = N_2 = N$, \cref{eq:n_subgroups} becomes,
\begin{equation*}
    c(N) = \sum_{d\cdot e = N} d 2^{\omega(e)},
\end{equation*}
where $\omega(e)$ counts the number of distinct prime factors of $e$~\cite{hampejs2014representingcountingsubgroupsgroup}.
If $N = p^\alpha$, for some prime $p$ and integer $\alpha$, the number of coefficient classes is then given by
\begin{align*}
    \sum_{d\cdot e = p^\alpha} d2^{\omega(e)}
    = \sum_{j = 0}^{\alpha} p^j 2^{\omega(p^{\alpha - j})} 
    =p^{\alpha}2^0  + \sum_{j = 0}^{\alpha - 1}p^j2^{1} 
    = 2 \frac{p^\alpha - 1}{p - 1} + p^{\alpha} 
    = \frac{p^{\alpha + 1} + p^\alpha - 2}{p - 1}\ .
\end{align*}
\end{example}


\section{Inversion Algorithms}
\label{sec:algs}

\delimitershortfall=0pt 

We now begin to address \cref{prob:alg} regarding recovery of integer signals and matrices.  
Let ${\bf x}$ be an integer signal of length $N$.
Without loss of generality, consider the minimal set of DFT coefficients, $\set{\tilde x_d : d \divs N}$.
\Cref{lem:amb1D,lem:dep1D} imply that ${\bf x}$ may be uniquely recovered from,
\begin{equation} \label{eq:ip_form}
    \begin{NiceArray}{[ccc]c[c]c}
        \uroot{N}{0\cdot d_1} & \cdots & \uroot{N}{(N-1)d_1} & \Block{3-1}{{\bf x} = } & \tilde{x}_{d_1} & \Block{3-1}{, \qquad {\bf x} \in \mathbb{Z}^N\ ,} \\
        \vdots & \ddots & \vdots & & \vdots \\
        \uroot{N}{0\cdot d_{\tau}} & \cdots & \uroot{N}{(N-1)d_{\tau}} & & \tilde{x}_{d_\tau}
    \end{NiceArray}
\end{equation}
where $d_1, \ldots, d_{\tau}$ are the divisors of $N$. 
Although the solution to the inverse problem is known to be unique, finding the solution is a highly nontrivial task. 
One clear difficulty arises from the lack of bounds on the range of integer values that ${\bf x}$ can take.
Many reconstruction methods, 
including branch-and-bound integer linear programming (ILP),
require such bounds to be efficient~\cite{Aardal2000}.
Even in the simplest case, when ${\bf x}$ is binary with each entry equal to 0 or 1, ILP is known to be NP-hard \cite{karp1975computational}.
Furthermore, since this is a feasibility problem with no objective function to optimize, the effectiveness of the method depends strongly on the rank of the linear system, which can be highly underdetermined~\cite{Lenstra1983}.
Defining an arbitrary objective function may improve or worsen the performance of branch-and-bound,
although it is difficult to know what objective is best~\cite{Aardal2000}.
Furthermore, solving \cref{eq:ip_form} directly with ILP can be unstable, as there may exist integer linear combinations of the roots of unity that come arbitrarily close, but do not exactly equal, the given DFT coefficients~\cite{myerson1986small,barber2023small}.

Our reconstruction methods build on the preceding theoretical results.  
In particular, instead of solving the entire linear system in \cref{eq:ip_form} at once, we decompose the problem into more manageable subproblems based on the decimation and subsignal concepts in \cref{def:dec} and \cref{def:subsig}.  
We demonstrate that this decomposition significantly decreases the overall search space. 

This section begins with a description of the general divide-and-conquer strategy. 
We then use this framework to develop an algorithm to reconstruct a one-dimensional integer signal from its minimal set of DFT coefficients.  
This algorithm generalizes the approach in \cite{pei2022binary1D}, which addresses the case where $N$ is a power of 2.  
The one-dimensional inversion algorithm also provides the foundation for the more general two-dimensional algorithm described afterwards.


\subsection{Inversion Strategy}
\label{ssec:strat}

Recall from \cref{def:dec} that, for $d\divs N$, ${\bf x}^{(N/d)}$ is the signal ${\bf x}$ decimated in frequency space to length $N/d$,
\begin{equation*}
    x^{(N/d)}_m = \sum_{n = 0}^{d - 1}x_{m + nN/d},
    \qquad \text{for}~0 \le m < N/d \ ,
    \end{equation*}
and, by \cref{lem:dec_in_freq}, satisfies $\tilde{x}^{(N/d)}_k = \tilde{x}_{dk}$ for all $0 \le k < N/d$. Therefore, knowledge of ${\bf x}^{(N/d)}$ for some divisor $d$ is equivalent to knowing $\tilde{x}_{dk}$ for all $0 \le k < N/d$.
We can then replace the corresponding row of \cref{eq:ip_form} with the equivalent set of $N/d$ equations given by the matrix form of frequency decimation in \cref{eq:matrix_form}, namely
\begin{equation} \label{eq:ip_dec}
    \begin{bmatrix}
        \idmat{N/d} & \cdots & \idmat{N/d}
    \end{bmatrix}
    {\bf x} = {\bf x}^{(N/d)} \ .
\end{equation}
This substitution is advantageous because it mitigates the stability and rank issues of \cref{eq:ip_form}.
Replacing the complex irrational coefficients of the linear system with integer coefficients improves numerical stability.  
Indeed, while integer linear combination of $N$th roots of unity are dense in the complex numbers~\cite{myerson1986small,Buhler2000dense,barber2023small}, any integer vector ${\bf x}$ that does not satisfy \cref{eq:ip_dec} will have an error of at least 1 (in any $p$-norm). 
Furthermore, when $d \ne N$, the substitution replaces a single equation with $N/d$ linearly independent equations, which increases the rank of the system.  
Note that this substitution is only valid when ${\bf x}$ is an integer signal.  

For each divisor $d$ of $N$,
if the decimated signal ${\bf x}^{(N/d)}$ is known, then performing the substitution in \cref{eq:ip_dec} can improve the ILP formulation.
However, note that if $d \divs d' \divs N$, then 
\begin{equation*}
    \begin{bmatrix}
        \idmat{N/d} & \cdots & \idmat{N/d}
    \end{bmatrix}
    {\bf x} = {\bf x}^{(N/d)}
   \ \text{ implies that } \ 
    \begin{bmatrix}
        \idmat{N/d'} & \cdots & \idmat{N/d'}
    \end{bmatrix}
    {\bf x} = {\bf x}^{(N/d')}
\end{equation*}
by \cref{lem:dec_in_freq}, since $\set{\tilde{x}_k : d' \divs k} \subseteq \set{\tilde{x}_k : d \divs k}$.
Therefore, it is sufficient to only consider the prime divisors of $N$ when performing the replacements of \cref{eq:ip_dec}.
This work shows that the ILP in \cref{eq:ip_form} may be written as the equivalent ILP, 
\begin{equation} \label{eq:ip_better}
    \begin{NiceArray}{[ccc]c[c]c}
        \idmat{N/p_1} & \cdots & \idmat{N/p_1} & \Block{4-1}{{\bf x} = } & {\bf x}^{(N/p_1)} & \Block{4-1}{,\qquad {\bf x}\in\mathbb{Z}^N\ ,}\\
        \vdots & \ddots & \vdots &  & \vdots & \\
        \idmat{N/p_{\omega}} & \cdots & \idmat{N/p_{\omega}} & & {\bf x}^{(N/p_{\omega})} \\
        \cmidrule(lr){1-3}
        \cmidrule(lr){5-5}
        \uroot{N}{0} & \cdots & \uroot{N}{(N - 1)} & & \tilde{x}_{1} \\
    \end{NiceArray}
\end{equation}
where $p_1, \ldots, p_\omega$ are the prime factors of $N$.
If both ${\bf x}$ and ${\bf y}$ in $\mathbb{R}^N$ satisfy the linear constraints in \cref{eq:ip_better}, then \cref{lem:dec_in_freq} implies that
$\tilde x_k = \tilde y_k$
for all $k$ such that some prime factor $p$ of $N$ divides $k$.
Since this is equivalent to the condition $\gcd(k, N) \ne 1$, the top $\omega$ blocks of the matrix in \cref{eq:ip_better} fix $N - \phi(N)$ DFT coefficients of ${\bf y}$.
The remaining last row implies that  $\tilde{y}_1 = \tilde{x}_1$. 
Note that as both ${\bf x}$ and ${\bf y}$ are real-valued, this last row also implies that $\tilde{y}_{-1}=\tilde{x}_{-1}$, as these DFT coefficients are the complex conjugates of $\tilde{y}_1$ and $\tilde{x}_1$. 
Therefore, $N - \phi(N) + 2$ DFT coefficients of ${\bf x}$ are fixed by linear constraints in \cref{eq:ip_better}.

In general, as ${\bf x}$ is real-valued, knowledge of the DFT coefficient $\tilde x_k$ also gives $\tilde{x}_{-k} = \tilde{x}_k^*$.
We can thus always rewrite the naive ILP in \cref{eq:ip_form} as,
\begin{equation} \label{eq:ip_conj}
    \begin{NiceArray}{[ccc]c[c]c}
        \uroot{N}{0 \cdot d_1} & \cdots & \uroot{N}{(N-1)d_1} & \Block{5-1}{{\bf x} = } & \tilde{x}_{d_1} & \Block{5-1}{, \qquad {\bf x} \in \mathbb{Z}^N\ ,}\\
        \uroot{N}{-0 \cdot d_1} & \cdots & \uroot{N}{-(N-1)d_1} & & \tilde{x}_{d_1}^* \\
        \cmidrule(lr){1-3}\cmidrule(lr){5-5}
        \vdots & \ddots & \vdots & & \vdots & \\
        \cmidrule(lr){1-3}\cmidrule(lr){5-5}
        \uroot{N}{0\cdot d_{\tau}} & \cdots & \uroot{N}{(N-1)d_{\tau}} & & \tilde{x}_{d_\tau}\\
        \uroot{N}{-0\cdot d_{\tau}} & \cdots & \uroot{N}{-(N-1)d_{\tau}} & & \tilde{x}_{d_\tau}^*
    \end{NiceArray}
\end{equation}

Since the DFT is bijective, the rank of this matrix is the number of unique rows.
For each divisor $d$, the two rows $\uroot{N}{\pm nd}$ of \cref{eq:ip_conj} are the same if and only if $d = N$ or $d = N/2$.
Therefore, the rank of the matrix in \cref{eq:ip_conj} is $2\tau(N) - 2$ if $N$ is even, and $2\tau(N) - 1$ if $N$ is odd.

In either \cref{eq:ip_better} or \eqref{eq:ip_conj}, the rank of the matrix describes the size of the ILP search space.
An exhaustive search may solve the ILP by iterating through integer combinations of the free variables
and checking whether the corresponding pivot variable values are integers.
As such, we assume that the dimension of the search space of any ILP we construct is the number of free variables of the linear system over $\mathbb{R}$, or equivalently the nullity of its matrix.
Note that while conjugate rows are always included in implementations of the ILPs, we will often omit them in the matrix formulations for clarity.
The corresponding conjugate linear equations can be assumed to be implicitly present.

The ILP in \cref{eq:ip_conj} will serve as a baseline for evaluating the performance of our algorithms.
As the rank of this matrix is 
$2\tau(N) - 1$ if $N$ is odd and $2\tau(N) - 2$ if $N$ is even, the dimension of its search space is given by its nullity $N-2\tau(N) + 1$ or $N-2\tau(N) + 2$. 
On the other hand, the DFT is bijective, so the rank of the matrix in \cref{eq:ip_better} is $N - \phi(N) + 2$ (where we have included the implicit conjugate last row).
Thus, the dimension of the search space of \cref{eq:ip_better} is $N-(N-\phi(N)+2)=\phi(N)-2$.
If $N$ is prime (and odd), $N-2\tau(N) + 1 = N - 3 = \phi(N)-2$, showing that the ILP formulations of \cref{eq:ip_form,eq:ip_better} are the same.
However, when $N$ is composite, this construction causes a significant reduction in the dimension of the search space.
For example, if $N = 30$, the search space associated with \cref{eq:ip_form} has dimenension $30 - 2\tau(30) + 2 = 16$.  
For \cref{eq:ip_better}, the dimension of the search space is reduced to $\phi(N)-2 = 6$. 
We now develop an algorithm which uses this approach iteratively to solve the inverse problem.


\subsection{1D Inversion} 

For each divisor $d$ of $N$, the vector ${\bf x}^{(N/d)}$ is an integer signal of length $N/d$,
so it may be recovered recursively by the same method used for ${\bf x}$.
\Cref{alg:1D} outlines this recursive approach to invert a one-dimensional integer signal from its minimal set of DFT coefficients.
Lines~\ref{ln:base1}-\ref{ln:base2} give the base case as recovery is trivial if $N = 1$.
Line~\ref{ln:rec1D_loop} iterates through the prime factors of $N$, 
recursively computing the integer signal decimated by each prime factor $p$ in frequency space.
Finally, Line~\ref{ln:rec1D_ip} uses the collected subproblem solutions to set up the ILP in \cref{eq:ip_better}, returning its unique solution as the recovered integer signal.

\begin{algorithm}[htb]
\caption{Recursive 1D Inversion}
\label{alg:1D}

\begin{algorithmic}[1]
    \Require DFT coefficients $\tilde{x}_{d}$ for each divisor $d$ of $N$
    \Procedure{Invert1D}{$\set{\tilde x_d : d \in \Call{Divisors}{N}}$}
        \If{$N = 1$} \label{ln:base1}
            \State \Return $\begin{bmatrix} \tilde{x}_0 \end{bmatrix}$ \label{ln:base2}
        \EndIf
        \ForAll{$p\in\Call{PrimeFactors}{N}$} \label{ln:rec1D_loop}
            \State ${\bf x}^{(N/p)} \gets \Call{Invert1D}{\set{\tilde{y}_{d} = \tilde x_{dp} :  d\in \Call{Divisors}{N/p}}}$ \label{ln:rec_call}
        \EndFor
        \State \Return the solution to the ILP, \label{ln:rec1D_ip}
        \Statex \makebox{%
            $\begin{NiceArray}{[ccc]c[c]c}
                \idmat{N/p_1} & \cdots & \idmat{N/p_1} & \Block{4-1}{{\bf x} = } & {\bf x}^{(N/p_1)} & \Block{4-1}{, \qquad {\bf x} \in \mathbb{Z}^N \ .}\\
                \vdots & \ddots & \vdots & & \vdots \\
                \idmat{N/p_{\omega}} & \cdots & \idmat{N/p_{\omega}} & & {\bf x}^{(N/p_{\omega})} \\
                \cmidrule(lr){1-3}
                \cmidrule(lr){5-5}
                \uroot{N}{0} & \cdots & \uroot{N}{(N - 1)} & & \tilde{x}_{1} \\
            \end{NiceArray}$
        }
    \EndProcedure
\end{algorithmic}
\end{algorithm}

While \cref{alg:1D} offers several improvements over a naive ILP or brute-force search solution, it suffers from computational redundancy due to repeated subproblems.
\Cref{fig:recursion} displays the subproblem recursion tree of \cref{alg:1D} when $N = 30$.
Observe that ${\bf x}^{(2)}$ appears as a subproblem of both ${\bf x}^{(6)}$ and ${\bf x}^{(10)}$. 
Thus, \cref{alg:1D} solves the subproblem for ${\bf x}^{(2)}$ twice.
The same is true for ${\bf x}^{(3)}$ and ${\bf x}^{(5)}$.
Note that the subproblem of size $1$ is solved 6 times by \cref{alg:1D}, but is not computationally relevant as this is the trivial base case.

\begin{figure}[htb]
\centering

\begin{tikzpicture}
\node (0) at (0,0) {${\bf x}$};
\node (1) at (-4.0,-1) {${\bf x}^{(6)}$};
\node (2) at (0.0,-1) {${\bf x}^{(10)}$};
\node (3) at (4.0,-1) {${\bf x}^{(15)}$};
\node (4) at (-5.0,-2) {${\bf x}^{(2)}$};
\node (5) at (-3.0,-2) {${\bf x}^{(3)}$};
\node (6) at (-1.0,-2) {${\bf x}^{(2)}$};
\node (7) at (1.0,-2) {${\bf x}^{(5)}$};
\node (8) at (3.0,-2) {${\bf x}^{(3)}$};
\node (9) at (5.0,-2) {${\bf x}^{(5)}$};
\node (10) at (-5.0,-3) {${\bf x}^{(1)}$};
\node (11) at (-3.0,-3) {${\bf x}^{(1)}$};
\node (12) at (-1.0,-3) {${\bf x}^{(1)}$};
\node (13) at (1.0,-3) {${\bf x}^{(1)}$};
\node (14) at (3.0,-3) {${\bf x}^{(1)}$};
\node (15) at (5.0,-3) {${\bf x}^{(1)}$};
\draw (3) -- (9);
\draw (7) -- (13);
\draw (2) -- (7);
\draw (4) -- (10);
\draw (0) -- (2);
\draw (1) -- (4);
\draw (9) -- (15);
\draw (6) -- (12);
\draw (0) -- (1);
\draw (2) -- (6);
\draw (3) -- (8);
\draw (8) -- (14);
\draw (5) -- (11);
\draw (1) -- (5);
\draw (0) -- (3);
\end{tikzpicture}
\vspace{-2.5pc}
\caption{
Recursion tree showing subproblem structure when \cref{alg:1D} is applied to a length $N = 30$ signal.}
\label{fig:recursion}
\end{figure}
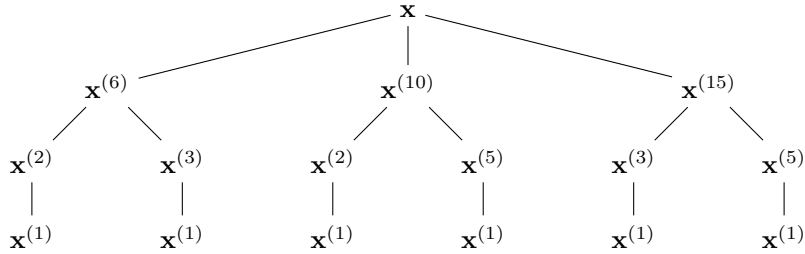

In general, this computational redundancy may be understood by considering the subgroup lattice of $\mathbb{Z}_N$.
The subgroup lattice uses set inclusion to partially order the subgroups of $\mathbb{Z}_N$.
Identifying the subgroups of $\mathbb{Z}_{N}$ with the divisors of $N$ gives a correspondence between the subgroup $\langle d\rangle$ of size $N/d$ and the subproblem ${\bf x}^{(N/d)}$.
The lattice structures in \cref{fig:z30} illustrate the redundancy of the size 2, size 3, and size 5 subproblems of a length 30 signal.
For example, both the subgroup $\langle 5\rangle$ of order $6$ and the subgroup $\langle 3\rangle$ of order $10$ on the third lattice level connect to the subgroup $\langle 15\rangle$ on the second level.
This parallels the dependence of both ${\bf x}^{(6)}$ and ${\bf x}^{(10)}$ on ${\bf x}^{(2)}$ in the subproblem lattice.

Using the correspondence between divisors of $N$ and subgroups of $\mathbb{Z}_N$, it is apparent that two subgroups are connected in the lattice if and only if their orders differ by a prime factor.
Applying this fact to the recursive step of \cref{alg:1D}, we see that each subproblem explicitly depends only on subproblems of the previous lattice layer. 
In particular, this implies that a dynamic programming implementation of \cref{alg:1D} can remove computational redundancy by first solving each of the subproblems on the bottom level, then the second level, and so on.
Since subgroups on the same lattice level are not related by inclusion, 
subproblems which share a lattice layer cannot depend on each other.
Thus, the order in which to solve the subproblems on each level is arbitrary.

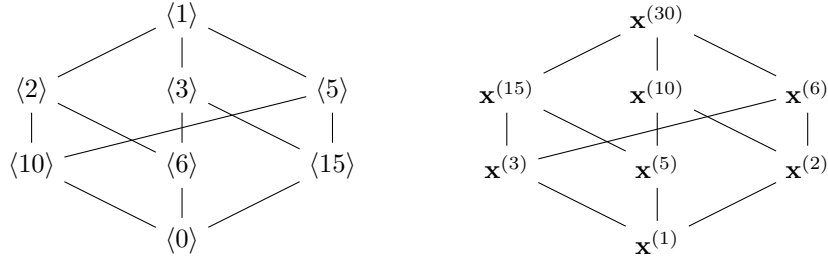
\begin{figure}[htb]
\centering

\begin{tikzpicture}
\foreach \n/\x/\y/\gen in {
30/0/0/0,
10/-2/1/10,
6/0/1/6,
15/2/1/15,
2/-2/2/2,
3/0/2/3,
5/2/2/5,
1/0/3/1%
} { \node (\n) at (\x,\y) {$\langle\gen\rangle$}; }
  \foreach \a/\b in {
15/5,
10/5,
6/2,
2/1,
3/1,
30/10,
5/1,
30/6,
6/3,
15/3,
30/15,
10/2%
  } { \draw (\a) -- (\b); }
\end{tikzpicture}%
\qquad\qquad
\begin{tikzpicture}
\foreach \n/\x/\y/\size in {
30/0/0/1,
10/-2/1/3,
6/0/1/5,
15/2/1/2,
2/-2/2/15,
3/0/2/10,
5/2/2/6,
1/0/3/30%
} { \node (\n) at (\x,\y) {${\bf x}^{(\size)}$}; }
  \foreach \a/\b in {
15/5,
10/5,
6/2,
2/1,
3/1,
30/10,
5/1,
30/6,
6/3,
15/3,
30/15,
10/2%
  } { \draw (\a) -- (\b); }
\end{tikzpicture}
\vspace{-2.5pc}
\caption{
{\bf Left}:
The Hasse diagram for the subgroup lattice of $\mathbb{Z}_{30}$.
Two subgroups $H \subsetneq K$ are connected if they are related by inclusion and there is no subgroup $H'$ satisfying $H \subsetneq H' \subsetneq K$.
The lattice is arranged into levels where each subgroup on any given level is only connected to subgroups on the levels immediately above and below it.
{\bf Right}: In contrast with \cref{fig:recursion},
the subproblem lattice of recovering a length 30 integer signal.
The identical lattice structure between the two diagrams suggest how the correspondence between subgroups and subproblems makes a dynamic programming algorithm efficient.
}
\label{fig:z30}
\end{figure}

\Cref{alg:memo_1D} implements one-dimensional inversion with dynamic programming.
The iteration order through the subproblems in Line~\ref{ln:memo1D_loop} sorts the divisors $N'$ of $N$ in increasing order.
Since $d \divs N'$ only if $d \le N'$, this ensures that the subproblems for all proper divisors of $N'$ have been solved prior to subproblem $N'$ itself. 
Therefore, each step has all decimated signals required to set up the ILP in \cref{eq:ip_better}.
Line~\ref{ln:memo1D_ip} solves the ILP, storing the result as the signal decimated to length $N'$ in ${\bf x}^{(N')}$ for future use.
After solving the subproblem for each divisor of $N$, the inverted signal is given as ${\bf x}^{(N)}$, which is returned in Line~\ref{ln:memo1D_ret}.

While \cref{alg:memo_1D} decreases the search space of each subproblem relative to an exhaustive search over the entire problem in \cref{eq:ip_form}, this comes at the potential cost of solving a larger number of ILPs.
To see that introducing the additional subproblems generally improves the overall complexity, consider the example of a length $N = 30$ binary signal ${\bf x}$. 
Recovering ${\bf x}^{(6)}$ (and thus also its subproblems ${\bf x}^{(1)}$, ${\bf x}^{(2)}$ and ${\bf x}^{(3)}$) is trivial.
Since uniqueness requires sampling $\tilde{x}_0$,
$\tilde{x}_1$,
$\tilde{x}_2$,
and $\tilde{x}_3$,
conjugate symmetry determines the remaining 
DFT coefficients, 
$\tilde{x}_4 = \tilde{x}_2^*$ and $\tilde{x}_5 = \tilde{x}_1^*$.
The binary constraint on ${\bf x}$ implies the entries of ${\bf x}^{(5)}$ are between 0 and 6.
The nullity of the ILP matrix for this subproblem is $\phi(5) - 2 = 2$.
The size of the search space for this subproblem is thus $7^2$. 
Similarly, the entries of ${\bf x}^{(10)}$ are between 0 and 3 and the entries of ${\bf x}^{(15)}$ are between 0 and 2. 
Since $\phi(10) - 2 = 2$ and $\phi(15) - 2 = 6$, 
these subproblems have size $4^{2}$ and $3^{6}$, respectively.
Finally, since $\phi(30) - 2 = 6$, the recovery of the binary signal ${\bf x}^{(30)}$ has a search space of size $2^{6}$. 
The sum of the search space sizes is $7^2 + 4^2 + 3^6 + 2^6 = 858$.
On the other hand, as $30-2\tau(3)+2=16$, an exhaustive search of \cref{eq:ip_form} would have a much larger search space of size $2^{16} = 65,536$.

\begin{algorithm}[htb]
\caption{Memoized 1D Inversion}
\label{alg:memo_1D}

\begin{algorithmic}[1]
    \Require DFT coefficients $\tilde{x}_{d}$ for each divisor $d$ of $N$
    \For{$N'\in\Call{Divisors}{N}$} \label{ln:memo1D_loop} \Comment{Includes $N$, iterate in increasing order.}
        \State $\set{p_1, \ldots, p_\omega} \gets \Call{PrimeFactors}{N'}$
        \State ${\bf x}^{(N')} \gets $ the solution to the ILP, \label{ln:memo1D_ip}
        \Statex \makebox{%
            $\begin{NiceArray}{[ccc]c[c]c}
                \idmat{N'/p_1} & \cdots & \idmat{N'/p_1} & \Block{4-1}{{\bf x} = } & {\bf x}^{(N'/p_1)} & \Block{4-1}{, \qquad {\bf x} \in \mathbb{Z}^{N'} \ .}\\
                \vdots & \ddots & \vdots & & \vdots & \\
                \idmat{N'/p_{\omega}} & \cdots & \idmat{N'/p_{\omega}} & & {\bf x}^{(N'/p_{\omega})} \\
                \cmidrule(lr){1-3}
                \cmidrule(lr){5-5}
                \uroot{N'}{0} & \cdots & \uroot{N'}{(N' - 1)} & & \tilde{x}_{N/N'} \\
            \end{NiceArray}$
        }
    \EndFor
    \State \Return ${\bf x}^{(N)}$ \label{ln:memo1D_ret}
\end{algorithmic}
    
\end{algorithm}


\subsection{2D Inversion}

We next develop a two-dimensional inversion algorithm by using the subsignal structure as a reduction to the one-dimensional problem.
\Cref{alg:2D} inverts an integer matrix ${\tt X}$ by individually inverting each of its subsignals.
Line~\ref{ln:2D_loop} iterates through representative frequencies $(k, l)$ of each DFT coefficient equivalence class.
Lines~\ref{ln:D1}-\ref{ln:D} compute the length of the $(k, l)$-subsignal as defined in \cref{eq:D_notation}.
Line~\ref{ln:2to1} inverts the $(k, l)$-subsignal by passing a minimal set of DFT coefficients to the one-dimensional inversion procedure.
Once the $(k, l)$-subsignal has been inverted, the corresponding DFT coefficients of ${\tt X}$ are updated.
Once all subsignals have been inverted, Line~\ref{ln:2D_ret} returns the inverted matrix ${\tt X}$. 
Note that in Lines~\ref{ln:2D_fillin} and \ref{ln:2D_ret}, we convert between real space and Fourier space. 
In Line~\ref{ln:2D_fillin}, $\tilde{{\bf x}}^{(k,l)}$ is computed from the reconstructed subsignal ${\bf x}^{(k,l)}$. 
To return ${\tt X}$ in Line~\ref{ln:2D_ret}, one must first take an inverse DFT of $\tilde{\tt X}$ which was compiled through the iterations of the algorithm.   
In a practical implementation, these steps are performed efficiently using fast Fourier transforms.
This takes $\bigo{D\log(D)}$ time in Line~\ref{ln:2D_fillin} and $\bigo{N_1N_2\log(N_1 + N_2)}$ time in Line~\ref{ln:2D_ret}~\cite{Cooley1965}.
As the fast Fourier transforms are much quicker than solving the NP-hard ILP at each iteration, 
conversion between real space and Fourier space makes a negligible impact on the runtime of \cref{alg:2D}.

\begin{algorithm}[htb]
\caption{High Level 2D Inversion}
\label{alg:2D}

\begin{algorithmic}[1]
\Require a representative DFT coefficient $\tilde{X}_{k, l}$ for each coefficient class
\ForAll{$(k, l) \in \Call{RepCoeffs}{N_1, N_2}$} \label{ln:2D_loop}
    \State $D_1 \gets N_1 / \gcd(k, N_1)$ \label{ln:D1}
    \State $D_2 \gets N_2 / \gcd(l, N_2)$ \label{ln:D2}
    \State $D \gets \lcm(D_1, D_2)$ \label{ln:D}
    \State ${\bf x}^{(k, l)} \gets \Call{Invert1D}{\set{\tilde{X}_{\lambda k, \lambda l} : \lambda \divs D}}$ \label{ln:2to1}
    \For{$0 \le \lambda < D$}
        \State $\tilde{X}_{\lambda k, \lambda l} \gets \tilde{x}_\lambda^{(k, l)}$ \label{ln:2D_fillin}
    \EndFor
\EndFor
\State \Return ${\tt X}$ \label{ln:2D_ret}
\end{algorithmic}

\end{algorithm}

In one dimension, the correspondence between (cyclic) subgroups of $\mathbb{Z}_N$ and sets of equivalent DFT coefficients revealed the lattice structure of subproblems.
We can generalize this subproblem structure to two dimensions by considering the correspondence in \cref{cor:coeff_count} between cyclic subgroups of $\mathbb{Z}_{N_1} \times \mathbb{Z}_{N_2}$ and DFT coefficient classes of an $N_1 \times N_2$ integer matrix.
\Cref{fig:z4xz6} shows the lattice of cyclic subgroups of $\mathbb{Z}_4 \times \mathbb{Z}_6$. 
As in the one-dimensional case, for $(k', l') \in \langle (k, l) \rangle$, the subgroups $\langle (k, l) \rangle$ and $\langle (k', l') \rangle$ are connected in the lattice if and only if $(k', l') = p(k, l)$ for some prime $p$ dividing the order of $(k, l)$.
Thus, solutions to connected subproblems lower in the lattice are required to set up \cref{eq:ip_better}.
The subproblem lattice reveals the computational redundancy arising from repeated subproblems in the naive divide-and-conquer implementation in \cref{alg:2D}.

\Cref{alg:2D} admits two types of redundancies.
The first comes from explicitly solving subproblems that are also solved implicitly by the one-dimensional inversion.
To invert a $4\times 6$ integer matrix, the outer for loop of \cref{alg:2D} solves a one-dimensional inversion problem for each vertex of the lattice in \cref{fig:z4xz6}.
Consider the iteration where the $(k, l) = (0, 1)$ subsignal is recovered.  
Within the one-dimensional inversion problem of this subsignal ${\bf x}^{(0, 1)}$ of length 6, we will recover $\big({\bf x}^{(0, 1)}\big)^{(3)}$, which is this signal decimated by 2. 
This decimated signal is exactly the subsignal ${\bf x}^{(0, 3)}$. 
However, \cref{alg:2D} also explicitly recovers the $(0, 3)$-subsignal at a different iteration of the loop in Line~\ref{ln:2D_loop} when $(k,l)=(0,3)$.

The second type of redundancy comes from overlapping subproblems in the recovery of the one-dimensional $(k, l)$- and $(k', l')$-subsignals of ${\tt X}$.
Consider $(k, l) = (1, 1)$ and $(k', l') = (1, 2)$. 
The cyclic subgroup lattice shows that both $\langle (k, l) \rangle$ and $\langle (k', l') \rangle$ contain $\langle (2, 2) \rangle=\set{(2, 2), (2, 4)}$.
Thus, recovering both the $(1, 1)$-subsignal and the $(1, 2)$-subsignal implicitly reconstructs the length 6 $(2, 2)$-subsignal.
As in the one-dimensional case, the lattice structure of the subproblems implies that a dynamic programming implementation can eliminate these redundancies.

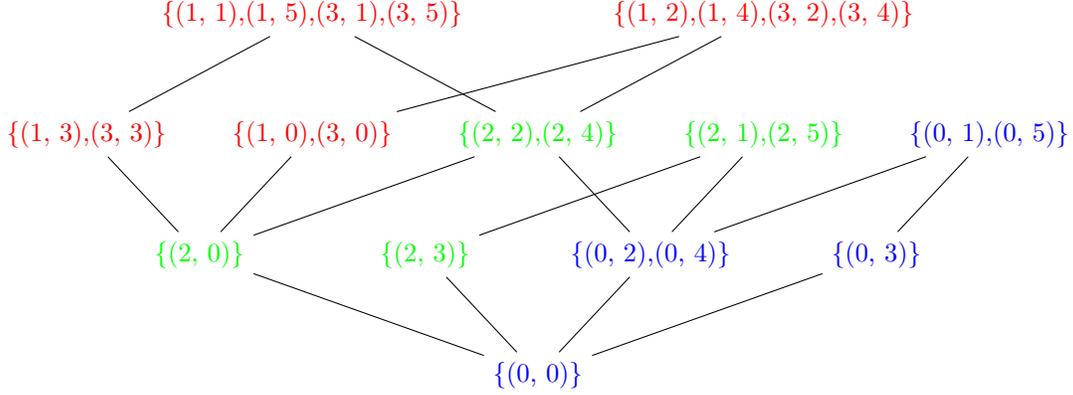
\begin{figure}[htb]
\centering
\def\fheight{0.8}

\begin{tikzpicture}
\foreach \n/\x/\y/\col/\group in {
0/0.0/0/blue!70!black/{(0, 0)},
8/-4.5/2/green!50!black/{(2, 0)},
11/-1.5/2/green!50!black/{(2, 3)},
2/1.5/2/blue!70!black/{(0, 2),(0, 4)},
3/4.5/2/blue!70!black/{(0, 3)},
1/-6.0/4/blue!70!black/{(0, 1),(0, 5)},
4/-3.0/4/red!70!black/{(1, 0),(3, 0)},
7/0.0/4/red!70!black/{(1, 3),(3, 3)},
9/3.0/4/green!50!black/{(2, 1),(2, 5)},
10/6.0/4/green!50!black/{(2, 2),(2, 4)},
5/-3.0/6/red!70!black/{(1, 1),(1, 5),(3, 1),(3, 5)},
6/3.0/6/red!70!black/{(1, 2),(1, 4),(3, 2),(3, 4)}%
} { \node[color=\col] (\n) at (\x,\y*\fheight) {$\set\group$}; }
\foreach \a/\b/\colo/\colr in {
1/2/red!70!black,
10/2/red!70!black,
10/8/red!70!black,
11/0/red!70!black,
2/0/red!70!black,
9/2/red!70!black,
8/0/red!70!black,
3/0/red!70!black,
5/7/red!70!black,
6/4/red!70!black,
5/10/red!70!black,
6/10/red!70!black,
4/8/red!70!black,
9/11/red!70!black,
1/3/red!70!black,
7/8/red!70!black%
} { \draw[color=black] (\a) -- (\b); }
\end{tikzpicture}
\vspace{-2.5pc}
\caption{
The Hasse diagram of the cyclic subgroup lattice of $\mathbb{Z}_{4} \times \mathbb{Z}_{6}$.
Only the generators of each subgroup are shown explicitly, as all other elements are contained in a proper subgroup lower in the lattice.
Color of vertices of $\langle(k, l)\rangle$ indicates iteration of loop in Line~\ref{ln:outer} of \cref{alg:memo_2D} where the $(k, l)$-subsignal is recovered ($D_1=1$ is blue, $D_1=2$ is green, and $D_1=4$ is red). 
}
\label{fig:z4xz6}
\end{figure}

Such an implementation is described in \cref{alg:memo_2D}. 
In Lines~\ref{ln:outer} and \ref{ln:inner}, the algorithm loops over pairs of divisors $D_1$ and $D_2$ of $N_1$ and $N_2$, respectively.
This orders the frequencies so that solutions to previous subproblems are available as needed.
Line~\ref{ln:loop_memo2D} iterates through the representative frequencies $(k, l)$ corresponding to the divisors $D_1$ and $D_2$. 
The order of this iteration is arbitrary.
In Line~\ref{ln:ip}, each $(k, l)$-subsignal is recovered by solving the ILP in \cref{eq:ip_better}.
Using the ILP solution, Line~\ref{ln:memo2D_fillin} updates the DFT coefficients of ${\tt X}$ from the  coefficient class represented by $(k, l)$.
Finally, Line~\ref{ln:memo2D_ret} returns the inverted matrix ${\tt X}$.
Note that \cref{alg:memo_2D} exhibits neither type of redundancy observed in \cref{alg:2D}.
\Cref{alg:memo_2D} does not use a one-dimensional inversion subroutine, 
which eliminates the first redundancy type.
The second type is handled by solving exactly one ILP for each coefficient class.

Next, we verify the correctness of the dynamic programming iteration order in \cref{alg:memo_2D}.
Fix a representative frequency $(k, l)$ and consider the iteration of the loops in 
Lines~\ref{ln:outer} and \ref{ln:inner}
that recovers the $(k, l)$-subsignal.
The condition in Line~\ref{ln:loop_memo2D} implies that this iteration satisfies $D_1 = N_1/\gcd(k, N_1)$ and $D_2 = N_2/\gcd(l, N_2)$.
For each prime factor $p$ of $D$, the $(k, l)$-subsignal decimated by $p$ in frequency space is given by
$({\bf x}^{(k, l)})^{(D/p)} = {\bf x}^{(pk, pl)}$.
To check that the $(pk, pl)$-subsignal was recovered prior to the current iteration, we consider two cases for the prime factor $p$.
If $p \divs N_1$, then $\gcd(pk, N_1) = p\gcd(k, N_1) = N_1/(D_1/p)$.
Since $D_1/p < D_1$, the $(pk, pl)$-subsignal was recovered at a previous iteration of the for loop in Line~\ref{ln:outer}.
If $p \ndivs N_1$, then $\gcd(pk, N_1) = \gcd(k, N_1)$.
Additionally, $p$ must divide $N_2$, so $\gcd(pl, N_2) = p\gcd(l, N_2) = N_2/(D_2/p)$.
Since $D_2/p < D_2$, the $(pk, pl)$-subsignal was recovered at the current iteration of the for loop in Line~\ref{ln:outer} and a previous iteration of the for loop in Line~\ref{ln:inner}. 
Note that since $D \divs \lcm(N_1, N_2)$, its prime factor $p$ must divide $N_1$ or $N_2$, so this argument shows that each required decimated subsignal is available.

\begin{algorithm}[htb]
\caption{Memoized 2D Inversion}
\label{alg:memo_2D}
\begin{algorithmic}[1]

\Require a representative DFT coefficient $\tilde{X}_{k, l}$ for each coefficient class
\For{$D_1 \in \Call{Divisors}{N_1}$} \label{ln:outer}
\For{$D_2 \in \Call{Divisors}{N_2}$} \label{ln:inner}
    \State $D \gets \lcm(D_1, D_2)$
    \State $\set{p_1, \ldots, p_\omega} \gets \Call{PrimeFactors}{D}$
    \ForAll{$(k, l) \in \Call{RepCoeffs}{N_1, N_2} : \gcd(k, N_1) = N_1/D_1, \gcd(l, N_2) = N_2/D_2$} \label{ln:loop_memo2D}
        \State ${\bf x}^{(k, l)} \gets$ the solution to the ILP, \label{ln:ip}
        \Statex \makebox{%
            $\begin{NiceArray}{[ccc]c[c]c}
                \idmat{D/p_1} & \cdots & \idmat{D/p_1} & \Block{4-1}{{\bf x} = } & {\bf x}^{(p_1k,p_1l)} & \Block{4-1}{,\qquad{\bf x} \in \mathbb{Z}^D\ .}\\
                \vdots & \ddots & \vdots & & \vdots & \\
                \idmat{D/p_{\omega}} & \cdots & \idmat{D/p_{\omega}} & & {\bf x}^{(p_{\omega}k,p_{\omega})} \\
                \cmidrule(lr){1-3}
                \cmidrule(lr){5-5}
                \uroot{D}{0} & \cdots & \uroot{D}{(D - 1)} & & \tilde{X}_{kl} \\
            \end{NiceArray}$
        }
        \For {$0 \le \lambda < D : \gcd(\lambda, D) = 1$} 
            \State $\tilde{X}_{\lambda k, \lambda l} \gets \tilde{x}_\lambda^{(k, l)}$ \label{ln:memo2D_fillin}
        \EndFor
    \EndFor
    \EndFor
\EndFor

\State \Return ${\tt X}$ \label{ln:memo2D_ret}
\end{algorithmic}

\end{algorithm}

Unlike the one-dimensional dynamic programming \cref{alg:memo_1D}, 
\cref{alg:memo_2D} does not completely explore each lattice layer before moving onto the next.
The double for loop through pairs of divisors performs more of a depth-first traversal of the subproblem lattice.
The cyclic subgroup lattice in \cref{fig:z4xz6} illustrates the subproblem order for inverting a $4 \times 6$ integer matrix.
The lattice vertices are colored according to which iteration of the outer for loop in Line~\ref{ln:outer} solves the subproblem.  
For example, the blue sublattice is recovered in the first iteration of the outer for loop, where $D_1 = 1$ (so the condition $\gcd(k, N_1) = N_1/D_1$ in Line~\ref{ln:loop_memo2D} implies $k = 0$).
The divisors of $6$ are $D_1 = 1, 2, 3, 6$, so the subproblems in this sublattice are solved in the order,
\begin{equation*}
    \set{(0, 0)}, \set{(0, 3)}, \set{(0, 2), (0, 4)}, \set{(0,1),(0,5)}.
\end{equation*}
At the next iteration of the outer for loop, $D_1 = 2$ and the green sublattice is recovered.
The red sublattice is recovered in the final iteration of the outer for loop, where $D_1 = 4$.

Note that \cref{alg:memo_2D} is not significantly more efficient than ILP for the $4 \times 6$ example considered above.
Since there are 12 cyclic subgroups in \cref{fig:z4xz6}, there are 12 classes of DFT coefficients. 
In particular, the two-dimensional analog of the ILP in \cref{eq:ip_form} has rank $2 \cdot 12 - 4 = 20$.
Since there are only four free variables in this system, a brute-force search is tractable.
To demonstrate the computational speedup of \cref{alg:memo_2D} in general, consider a $30 \times 30$ binary matrix.
A $30 \times 30$ integer matrix has 140 coefficient equivalence classes, so we will not show the size of every subproblem.
As one example, consider recovery of the $(2, 6)$-subsignal.
For the $(k, l) = (2, 6)$ frequency, we have $D_1 = N_1/\gcd(k, N_1) = 15$, $D_2 = N_2/\gcd(l, N_2) = 5$, and $D = \lcm(5, 15) = 15$.
Each entry of the $(2, 6)$-subsignal is thus a sum of $N_1N_2/D = 60$ entries of ${\tt X}$, and thus has 61 possible values.  As the nullity of the relevant matrix is $\phi(15) - 2 = 6$, we have a total search space size of $61^{6}$ = 5.15e10 for this subproblem.
Repeating this computation for each of the 140 subproblems yields a total search space size of 4.88e11.
The two-dimensional analogue of the ILP in \cref{eq:ip_form} has a search space size of $2^{900 - 140 \cdot 2 + 4} = 2^{624}$ = 6.96e187.

\delimitershortfall=5pt 


\section{Numerical Reconstructions} \label{sec:numerics}

We now demonstrate the substantial advantage of our divide-and-conquer approach numerically.
While \cref{alg:memo_2D} is an improvement over naive ILP approaches, 
each iteration still involves solving a smaller ILP. 
We thus begin by presenting two approaches for solving the ILPs.

\subsection{Implementation} \label{sec:implement}

We first solve the subproblem ILPs directly by supplying the feasibility problem \cref{eq:ip_better} to Gurobi's mixed integer programming branch-and-cut optimizer~\cite{gurobi,padberg1991}.
Since our inverse problem ILP is a feasibility problem,
we set the Gurobi \texttt{MIPFocus} parameter to prioritize searching for feasible solutions. 
For each subproblem,
we set the Gurobi \texttt{WorkLimit} parameter to 100,
which gives a deterministic time limit.
In all of our tests,
we assume that we know \textit{a priori} that the entries of the image are bounded by $0 \le X_{mn} \le L$, for some known bound $L$.
Our ILP implementation of \cref{alg:memo_2D} applies this assumption to all ILP subproblems,
using the bound $L=N_1N_2/D$
when solving for the $(k, l)$-subsignal of length $D$.

Our second approach solves the subproblem ILPs using a lattice problem reformulation.
For a set of $d$ linearly independent vectors $\set{{\bf b}_1, \dots, {\bf b}_d} \subseteq \R^n$,
the collection of integer linear combinations forms a discrete subgroup of $\R^n$,
\begin{equation}
    \mathcal{L}({\bf b}_1, \dots, {\bf b}_d)
    = \set*{\sum_{i = 1}^d \alpha_i{\bf b}_i \given {\bm \alpha} \in \Z^d}.
\end{equation}
The set $\mathcal{L}$ is called a $d$-dimensional lattice,
and the vectors $\set{{\bf b}_1, \dots, {\bf b}_d}$ are a basis for $\mathcal{L}$.
As $\mathcal{L}$ is discrete,
it has a shortest nonzero vector,
although finding the exact shortest vector is a computationally hard problem~\cite{ajtai1998shortest,micciancio2001shortest}.
The Lenstra-Lenstra-Lov\'asz (LLL) algorithm finds a basis of short,
nearly orthogonal vectors for an input lattice,
which provides a sufficiently short approximation to the shortest vector for many practical applications~\cite{nguyen2006average},
and crucially runs in polynomial-time~\cite{stehle2009floating}.

Consider a general feasibility ILP of finding a vector ${\bf x} \in \Z^n$,
such that ${\tt A}{\bf x} = {\bf b}$, 
where ${\tt A}$ is an $m \times n$ matrix of constraints.
We reformulate this problem as a lattice problem by constructing the basis given by the columns of the $(n + m) \times (n + 1)$ block matrix,
\begin{equation} \label{eq:lattice_basis}
    {\tt B} = \begin{bmatrix}
        \idmat{n}& {\tt 0} \\
        \beta {\tt A} & -\beta{\bf b}
    \end{bmatrix},
\end{equation}
where the numerical parameter $\beta$ is taken to be large.
Any vector in $\mathcal{L}({\tt B})$ takes the form ${\tt B}\begin{bmatrix}
    {\bm \alpha} & \gamma
\end{bmatrix}^T
= \begin{bmatrix}
    {\bm \alpha} & \beta({\tt A}{\bm\alpha} - \gamma{\bf b})
\end{bmatrix}^T$,
where ${\bm \alpha} \in \Z^n$ and $\gamma \in \Z$ are the lattice coefficients.
As $\beta$ is large, a short lattice vector must satisfy 
${\tt A}{\bm \alpha} = \gamma{\bf b}$, since otherwise the final $m$ coordinates will dominate its norm.  
If $\gamma = 1$,
this gives a solution ${\bm \alpha}$ to the ILP.
In the case when ${\bf A}$ and ${\bf b}$ are integer-valued, the work of \cite{Aardal2000} derives sufficient bounds on $\beta$ for the LLL algorithm to return such a solution. 
When applied to the basis in \cref{eq:lattice_basis},
the LLL runtime is bounded by,
\begin{equation}
\label{eq:lll_runtime}
   \bigo{n^4(n + m)(n + \log B)\log B} ,  
\end{equation}
where $B$ bounds the Euclidean norms of the columns of ${\tt B}$~\cite{nguyen2009lll}.

We remark that branch-and-cut and the lattice formulation provide two distinct methods to solve a given feasibility ILP. 
In general,
it is difficult to know which method will perform better,
and each has its own strengths and weaknesses.
The size of the branch-and-cut tree scales with the magnitude of bounds on the integers,
making the performance of direct ILP methods highly dependent on these bounds.
On the other hand,
LLL does not explicitly exploit variable bounds,
and its performance is relatively indifferent to their actual magnitude.
In our setting, however, the density of the cyclotomic integers in $\C$ makes the ILP constraints unstable.  Therefore, scaling by a large value of $\beta$ amplifies errors in the DFT coefficient data, making the lattice formulation less robust to noise.  
Thus,
although LLL has polynomial runtime and often performs well in practice, direct ILP methods may be preferable 
when tight variable bounds are available and the data precision is limited~\cite{levinson2023}.

\subsection{Results}

We first investigate the limitations of standard ILP algorithms.
\Cref{tab:rect} presents numerical results comparing \cref{alg:memo_2D} with naively solving the ILP in \cref{eq:ip_form}.
For several different image sizes $N_1 \times N_2$,
the table reports
the maximal subsignal size $N = \lcm(N_1, N_2)$,
the number of DFT coefficients in a minimal data set,
and, for each algorithm, 
the recovery rate of 20 random binary matrices and the average runtime on a single core of a standard workstation.
The middle columns ``Naive (ILP)" and ``\cref{alg:memo_2D} (ILP)" show the limitations of both ILP-based approaches, 
which become slow for even moderately sized images.  
However, with the chosen ILP parameters, \cref{alg:memo_2D} successfully reconstructs its largest image ($24\times24$) in about 90 minutes on a single core of the standard workstation, whereas the naive ILP approach fails.
\Cref{alg:memo_2D} also performs substantially better on rectangular instances such as  $9\times 11$, $10\times 15$, and $10\times 16$.

\begin{table}[htbp]
\centering

\crefname{algorithm}{Alg}{Alg}
\setlength{\tabcolsep}{3pt}

\begin{tabular}{c@{\hspace{15pt}}S[table-format=2]S[table-format=3]|S[table-format=3]S[table-format=4.1]|S[table-format=3]S[table-format=5.1]|S[table-format=3]S[table-format=3.2]|S[table-format=3]S[table-format=0.2]}
\multicolumn{3}{c|}{Algorithm}& \multicolumn{2}{c|}{Naive (ILP)}& \multicolumn{2}{c|}{\cref{alg:memo_2D} (ILP)}& \multicolumn{2}{c|}{Naive (LLL)}& \multicolumn{2}{c}{\cref{alg:memo_2D} (LLL)}\\
{$N_1 \times N_2$} & \multicolumn{1}{c}{$N$} & \multicolumn{1}{c|}{Coeffs, \#} & \multicolumn{1}{c}{Rec, \%} & \multicolumn{1}{c|}{$t$, sec} & \multicolumn{1}{c}{Rec, \%} & \multicolumn{1}{c|}{$t$, sec} & \multicolumn{1}{c}{Rec, \%} & \multicolumn{1}{c|}{$t$, sec} & \multicolumn{1}{c}{Rec, \%} & \multicolumn{1}{c}{$t$, sec}\\
\midrule\midrule
$\phantom{0}9 \times 11\phantom{}$ & 99 & 6 & 0 & 585.4  & 85 & 709.8  & 40 & 0.18  & 80 & 0.24 \\
$\phantom{}10 \times 15\phantom{}$ & 30 & 28 & 95 & 146.3  & 100 & 15.9  & 100 & 3.06  & 100 & 0.03 \\
$\phantom{}10 \times 16\phantom{}$ & 80 & 20 & 0 & 353.3  & 90 & 350.0  & 0 & 0.96  & 100 & 0.29 \\
$\phantom{}11 \times 11\phantom{}$ & 11 & 13 & 100 & 1.1  & 90 & 5930.8  & 100 & 0.56  & 100 & 0.01 \\
$\phantom{}12 \times 12\phantom{}$ & 12 & 50 & 100 & 0.0  & 100 & 0.1  & 100 & 0.69  & 100 & 0.01 \\
$\phantom{}12 \times 30\phantom{}$ & 60 & 60 & 0 & 337.4  & 5 & 2580.5  & 0 & 27.62  & 100 & 0.38 \\
$\phantom{}13 \times 13\phantom{}$ & 13 & 15 & 100 & 8.6  & 0 & 1245.6  & 100 & 1.51  & 100 & 0.01 \\
$\phantom{}16 \times 16\phantom{}$ & 16 & 46 & 95 & 66.0  & 100 & 399.9  & 100 & 15.82  & 100 & 0.02 \\
$\phantom{}17 \times 17\phantom{}$ & 17 & 19 & 0 & 250.6  & 0 & 963.0  & 100 & 5.81  & 100 & 0.02 \\
$\phantom{}18 \times 18\phantom{}$ & 18 & 68 & 100 & 24.7  & 100 & 42.7  & 100 & 39.82  & 100 & 0.04 \\
$\phantom{}21 \times 21\phantom{}$ & 21 & 45 & 0 & 260.6  & 0 & 863.5  & 0 & 57.74  & 100 & 0.07 \\
$\phantom{}24 \times 24\phantom{}$ & 24 & 110 & 0 & 284.0  & 100 & 5316.6  & 0 & 235.52  & 100 & 0.11 \\
$\phantom{}30 \times 30\phantom{}$ & 30 & 140 & 0 & 162.9  & 0 & 780.8  & 0 & 160.14  & 100 & 0.33 \\
\end{tabular}

\crefname{algorithm}{Algorithm}{Algorithm}
\setlength{\tabcolsep}{6pt}

\caption{
For various matrix dimensions $N_1 \times N_2$:
the number of DFT coefficients required for uniqueness,
the percentage of test signals recovered, and
the average time spent
for each inversion algorithm implementation.
For each matrix size, the test set consisted of 20 random binary matrices.
}

\label{tab:rect}
\end{table}

In contrast,
\cref{alg:memo_2D} performs worse than the naive ILP approach for smaller square dimensions in \cref{tab:rect} ($11\times 11$, $12 \times 12$, and $13 \times 13$).
In these cases, the naive method
benefits from reduced computational overhead, 
as it solves a single ILP with strict binary constraints on the unknowns.  
We note that both algorithms performed particularly well for the $12 \times 12$ images,
because the corresponding ILPs are only mildly underdetermined.  A $12\times12$ image has 74 independent DFT coefficients after accounting for conjugate symmetry, of which the naive algorithm samples 50.  Likewise, \cref{alg:memo_2D} recovers each length 12 subsignal from 6 of its 7 independent DFT coefficients.  In both cases, only a small number of unknown coefficients remain, making the resulting ILPs comparatively easy to solve.
For the $13 \times 13$ case,
\cref{alg:memo_2D} particularly struggled due to the 14 independent subproblems of dimension 13, 
each with looser bounds on the variables (between 0 and 13).
We hypothesize that these relaxed constraints are less effective at enabling the computational reductions typically achieved by ILP heuristics and cutting plane methods. 
\Cref{tab:L} confirms this empirically by investigating the dependence of each algorithm's performance on the integer bound $L$.
We see that the naive ILP method is faster for $7 \times 7$ images with small integer entries,
but \cref{alg:memo_2D} quickly out-scales the naive method as $L$ increases.

\begin{table}[htbp]
    \centering

    \crefname{algorithm}{Alg}{Alg}
\begin{tabular}{c|S[table-format=4.2, table-space-text-post = {*}]S[table-format=2.2, table-space-text-post = {*}]S[table-format=1.4, table-space-text-post = {*}]S[table-format=1.4, table-space-text-post = {*}]}
L & \multicolumn{1}{c}{Naive (ILP)} & \multicolumn{1}{c}{\cref{alg:memo_2D} (ILP)} & \multicolumn{1}{c}{Naive (LLL)} & \multicolumn{1}{c}{\cref{alg:memo_2D} (LLL)}\\
\midrule\midrule
1 & 0.02 & 0.25 & 0.0165 & 0.0029 \\
2 & 0.18 & 1.12 & 0.0166 & 0.0030 \\
3 & 0.58 & 3.53 & 0.0164 & 0.0029 \\
4 & 2.11 & 11.09 & 0.0165 & 0.0030 \\
5 & 12.71 & 23.37 & 0.0165 & 0.0029 \\
6 & 172.20 & 54.50 & 0.0164 & 0.0029 \\
7 & 1817.27 & 95.49 & 0.0165 & 0.0029 \\
\end{tabular}
\crefname{algorithm}{Algorithm}{Algorithm}

    \caption{For increasing values of $L$:
    the average time taken by each inversion implementation to recover a test signal of size $7 \times 7$ with integer entries distributed uniformly in $[0, L]$.
    For each $L$, the test set contained 20 random matrices.
    Note that the ILP parameter \texttt{WorkLimit} was increased to 5000 for this numerical test.
    }
    \label{tab:L}
\end{table}

We emphasize that the nature of the uniqueness results in \cref{thm:dep2D,thm:amb2D}, together with the subgroup lattice structure, makes the algorithm’s performance highly dependent on the image dimensions $N_1$ and $N_2$. In the numerical experiments, the number of DFT measurements in the minimal data set naturally correlates with recovery success.
For example, consider the cases $N_1 \times N_2 = 16 \times 16, 17 \times 17,$ and $18 \times 18$.
At 18.0\% and 21.0\% respectively,
the $16 \times 16$ and $18 \times 18$ images have relatively larger data sets than the $17 \times 17$ at 15.2\%.
Indeed, both the naive and divide-and-conquer ILP approaches reliably recovered the $16 \times 16$ and $18 \times 18$ test images,
while failing to recover any of the $17 \times 17$ tests within the time limit.

Comparing recovery times for \cref{alg:memo_2D} using ILP,
the $18 \times 18$ test images were recovered much faster than the $12 \times 30$ images,
despite having similar sizes and a similar number of subproblems.
We attribute this difference to the sizes of the ILP search space in each case.
A binary $12 \times 30$ image has 360 pixels,
so each entry of the length 60 subsignals have entries between 0 and 6.
Since $\phi(60) - 2 = 22$,
the search space associated with recovering these maximal subsignals has size $6^{22} = $ 1.31e17.
On the other hand, 
for a binary $18 \times 18$ image with $324$ entries,
the largest subsignal of length 18 has a search space size of $18 ^ {4} = $ 1.05e5.
Similarly,
we can explain the failure of \cref{alg:memo_2D} using ILP to recover any $21 \times 21$ test matrix while successfully recovering all $24 \times 24$ test signals by the relative search space sizes of $21^{12} = $ 7.36e15 and $24^{6} = $ 1.91e8, respectively, for each maximal subproblem. 

In our testing,
we observed a strong correlation between the number of branch-and-cut nodes explored by the ILP solver and the size of the search space,
although heuristics and cutting planes allow the solver to find the unique solution by exploring only a fraction of the search space.
These ILP heuristics appear most effective for smaller values of $L$, 
as evidenced by \cref{tab:L}. 
This also explains why \cref{alg:memo_2D} successfully inverted some $12\times 30$ test images but none of the size $21\times 21$ tests. 
Although the largest subproblems of the latter case have a smaller search space, they also have a looser bound on the unknowns ($L=21$ as opposed to $L=6$).

Motivated by the overall limitations of these ILP results, we employed the reformulation of \cref{eq:ip_better} as a lattice reduction problem.
We apply the LLL algorithm to the associated lattice bases in \cref{eq:lattice_basis} for each of the ILPs in \cref{eq:ip_form,eq:ip_better}~\cite{Lenstra1982}. 
Our numerical implementations are available at \cite{intvert}.
The recovery data reported in the last two columns of \cref{tab:rect}
demonstrate the immense speed and scaling benefits of implementing \cref{alg:memo_2D} with LLL.
In particular, while all four implementations successfully recovered every $18\times 18$ test image, 
the LLL-based \cref{alg:memo_2D} ran over 500 times faster than the other methods. 
Furthermore, this implementation exactly recovered the larger test images of sizes $21\times 21$, $24\times 24$, and $30\times 30$ in under one second, 
while all other implementations took much longer and/or failed.
As one outlier, 
note that the relatively prime dimensions of the $9 \times 11$ images favor ILP over LLL, where there is a subproblem of dimension 99 with $L=1$.
As previously noted,
LLL does not incorporate this tight integer bound and is sensitive to the size of the system,
while ILP is generally well-suited to the small $L$ value~\cite{levinson2023}.
This phenomenon also appears in \cref{tab:L},
where the last two columns demonstrate that the LLL implementations are relatively agnostic of the integer bound $L$.

In \cref{tab:rect}, 
our artificially imposed time limit accounted for the ILP failures,
and we believe that,
given sufficient time,
the algorithms would eventually recover the correct signal. 
On the other hand,
our LLL implementation has no time limit, and incorrect recoveries are caused by the limited stability of double precision data.

The largest ILP system for a subproblem of \cref{alg:memo_2D} has dimension $\lcm(N_1, N_2)$ whereas the naive algorithm solves a single ILP of dimension $N_1N_2$.  Since 
the LLL runtime bound \cref{eq:lll_runtime} scales as $n^6$, where $n$ is the dimension of the system, this difference in problem size explains the large runtime gap between the last two columns of \cref{tab:rect}.
Moreover, the naive LLL approach also becomes noticeably less reliable as the system dimension increases.  We hypothesize that this loss of stability is due to the  LLL approximation factor,
which grows exponentially with the dimension of the ILP system,
compounding the inherent instability of the inverse problem~\cite{kalbach2024lllalgorithmlatticebasis}.
Hence, the dynamic programming structure of \cref{alg:memo_2D} is a substantial improvement by reducing the size of the ILP subproblems, thereby improving both the runtime and stability of the LLL solver.  Combining this with the relative insensitivity of LLL to the integer bounds facilitates the immense success observed in \cref{tab:rect}.

Going beyond the moderate image sizes tested in \cref{tab:rect}, we conclude this section with a few larger, structured, nonrandom image reconstructions using \cref{alg:memo_2D} with LLL.
As an initial binary example, \cref{fig:qr} shows a $45\times45$ QR code, which is the Version 7 standard~\cite{tiwari2016}. 
We sampled a minimal set of 119 out of 2025 total DFT coefficients (5.9\%).  
The left image in \cref{fig:qr} shows the zero-filled reconstruction in real space, where all unsampled DFT coefficients are set to 0. 
Using double precision in the data, \cref{alg:memo_2D} quickly recovered the QR code exactly in 1.5 seconds. 
Note that this reconstruction did not take advantage of the fixed structure and  built-in error correction present in the QR code. 

\begin{figure}[htb]
    \centering
    \includegraphics[width=.97\textwidth]{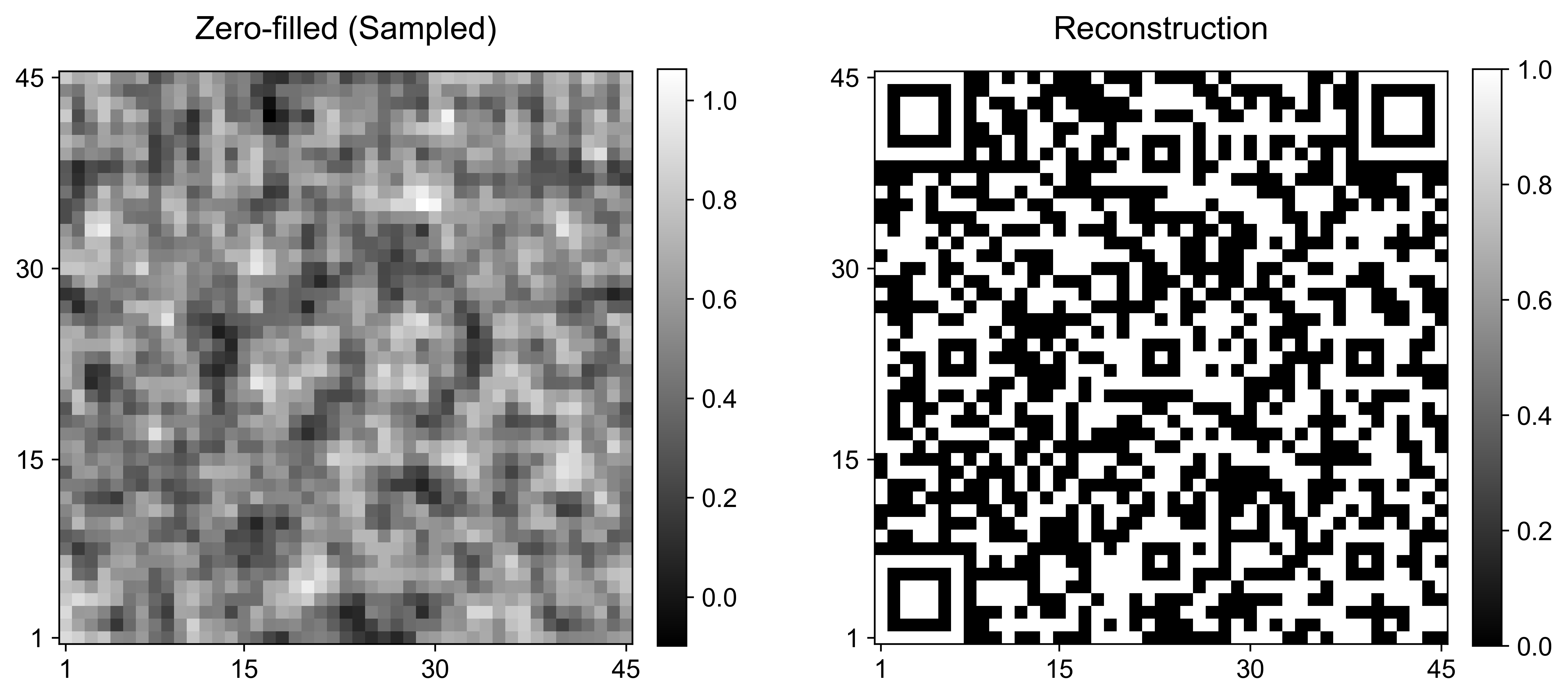}
    
    \caption{
    Using a $45 \times 45$ QR code containing the message ``Hello World!'' as a model image,
    {\bf Left}: the zero-filled reconstruction from a minimal set of sampled DFT coefficients;
    {\bf Right}: the reconstruction by \cref{alg:memo_2D} from the minimal spectrum (coincides exactly with model).
    }
    \label{fig:qr}
\end{figure}

Next, we consider a non-binary image with a larger integer range. 
We begin with the classic $256\times256$ `camera man' image, where the integer values range from 0 to 255~\cite{schreiber2005image}. 
This resolution was too large for our algorithm to handle directly, so for practicality (and taking into consideration the number of subproblems), we rescaled the image to $60\times60$. 
To ensure feasibility at double precision, 
we also rescaled the intensity values to be between 0 and $L=19$.  
The top right image in \cref{fig:images} shows this rescaled model image. 
The top left image shows the corresponding zero-filled reconstruction, obtained with the minimal 350 sampled DFT coefficients (9.7\%). 
With these parameters, \cref{alg:memo_2D} recovered the model image exactly in 6.5 seconds.  

This reconstruction pushes the limit of double precision. 
As previously discussed, 
the density of the cyclotomic integers makes the DFT coefficient constraints in \cref{eq:ip_form,eq:ip_better} unstable. 
As this instability worsens with increasing size, the importance of precision becomes greater for larger examples with LLL.  Therefore, to successfully reconstruct larger images we used multiple precision floating-point computations to ensure the unique solution was recovered.
Moving to higher precision lets us expand the integer range $L$ and work with larger images. 

In the middle row of \cref{fig:images},
we consider the original range of integer values by setting $L = 255$.
This reconstruction failed at double precision,
but succeeded in 9.1 seconds when the precision was extended to 25 digits.
Finally,
we tackled a larger image,
increasing the size to $210\times210$,
and sampling a minimal set of 1260 DFT coefficients (2.86\%).
We again consider the original range of integer values with $L = 255$.
To make this reconstruction feasible,
we had to extend the precision to 100 digits.
The reconstruction was again exact,
with no reconstruction error
and took 
187.2 minutes.

\begin{figure}
    \centering
    \includegraphics[width = 0.93\textwidth]{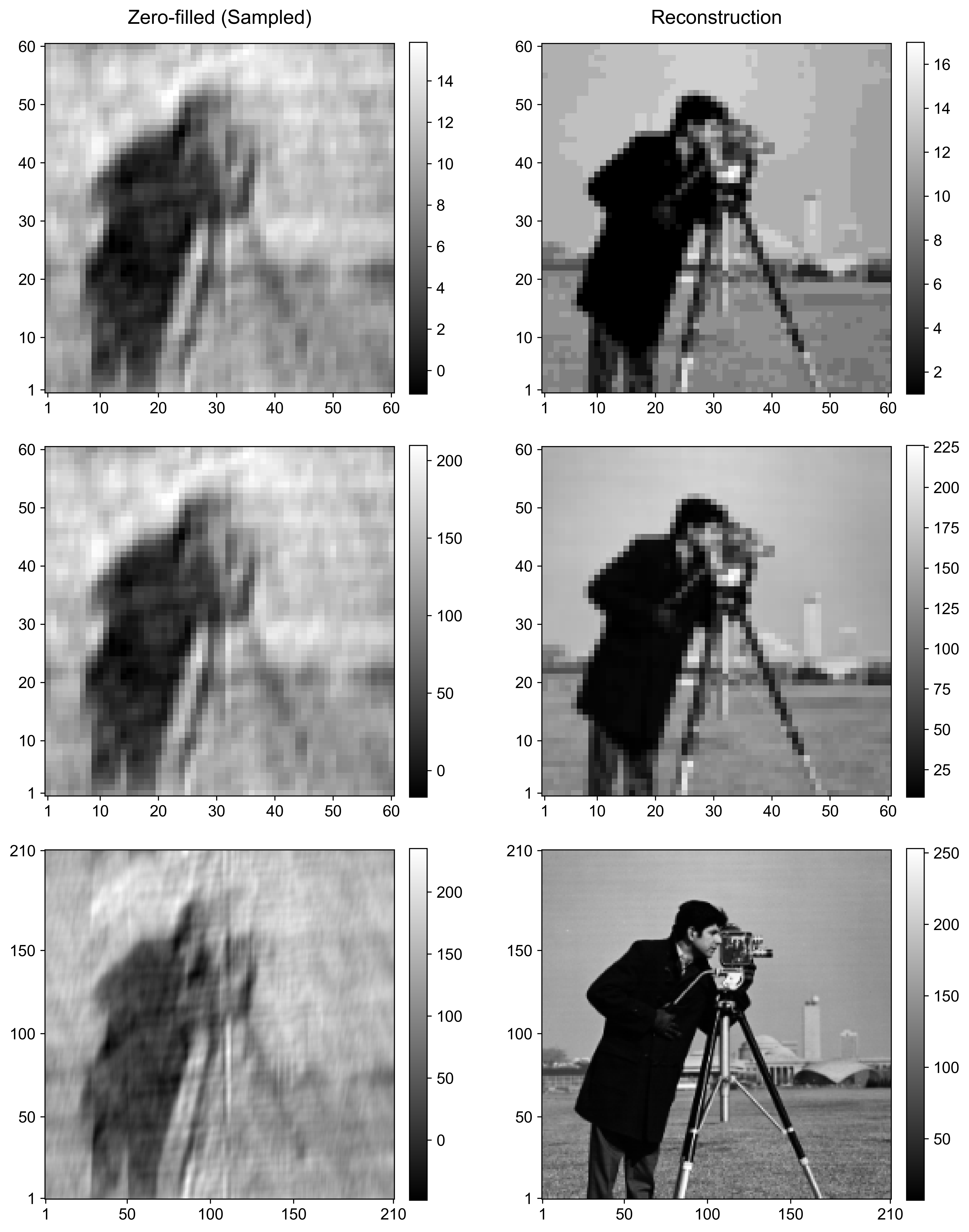}
    
    \caption{
    For three models, the left column has the zero-filled reconstruction from a minimal set of sampled DFT coefficients; the right column is the reconstruction by \cref{alg:memo_2D} from the minimal spectrum (which coincides exactly with model). 
    {\bf Top}: A $60\times60$ model with $L=19$ using double precision in the data.
    {\bf Middle}: A $60\times60$ model with $L=255$ using 25 digits of precision in the data. 
    {\bf Bottom}: A $210\times210$ model with $L=255$ using 100 digits of precision in the data.
    }
    \label{fig:images}
\end{figure}


\section{Discussion}
The results of this work demonstrate the strength of integer-valued priors, 
as they enable perfect recovery from a limited set of DFT coefficients.  
Although reconstruction chances depend on the amount of data in the minimal set,
our lattice method typically succeeded with fewer than 15\% of the coefficients measured. 
The LLL reconstruction demonstrates reasonable speed and accuracy,
but often demand high precision DFT data.
While elevated precision permits recovery of larger images, 
moving to lower precision reduces algorithm reliability.
Thus,
the presence of noise in the data can make the methods unreliable.

Additional numerical testing indicated that reconstruction of random $30\times30$ binary images was reliable with 5-6 decimal digits of precision, 
but frequently failed with fewer than 5 digits. 
While this paper focuses on reconstruction from a minimal set of noiseless coefficients,
important questions remain regarding how to stably scale to larger images and noisy data. 
A natural and practical direction is to incorporate additional data beyond the minimal sampling,
while remaining highly underdetermined.
A detailed analysis of stability, 
particularly with limited additional sampling, is warranted.  
Moreover, the LLL method depends on several parameters, and a study of their optimal selection would be valuable for practical applications.

Incorporating algorithm elements from related works could further improve the stability and speed of our methods.
Combining integer linear programming with lattice methods shows promising results,
especially when tight integer bounds are available~\cite{levinson2023}.
With an appropriate sampling strategy,
an FFT-like adaptation of the algorithms  benefits from the inclusion of many redundant DFT coefficients~\cite{pei2022binary1D}.
Additionally, the partial order on the subproblems of \cref{alg:memo_1D,alg:memo_2D} prescribes efficient ways to parallelize the algorithm implementations.
While our results demonstrate the effectiveness of integer-valued priors, a natural direction for future work is to combine these priors with other common ones, such as sparsity or connectivity.
Finally, we note that the reduction techniques developed in this work for theoretical results in \cref{thm:dep2D,thm:amb2D} and \cref{alg:memo_2D} extend naturally to higher-dimensional integer signals.

\appendix

\section{Proof of \texorpdfstring{\cref{lem:subsig}}{Lemma 3.4}} \label{ap:subsig}

\begin{proof}

We first show the existence of an $m$ and $n$ that satisfy \cref{eq:to_satisfy} for each $0 \le j < D$. 
This will show that each entry $x_j^{(k,l)}$ contains a contribution from at least one entry of ${\tt X}$.  Let $d_1=\gcd(k,N_1)=N_1/D_1$ and $d_2=\gcd(l,N_2)=N_2/D_2$, and recall the definitions from \cref{eq:D_notation}.
Repeatedly applying the relationship  $ab = \gcd(a,b)\lcm(a,b)$ and using the definition of $d_1$ and $d_2$ yields,
\begin{equation*}
    d = \frac{N}{D}
    = \frac{\lcm(N_1, N_2)}{\lcm(D_1, D_2)}
    = \frac{N_1N_2\gcd(D_1, D_2)}{D_1 D_2\gcd(N_1, N_2)} 
    = \frac{d_1d_2\gcd(N_1/d_1, N_2/d_2)}{\gcd(N_1, N_2)} \ .
\end{equation*}
Next, repeatedly applying the fact, 
$c\gcd(a, b) = \gcd(ca, cb)$,
to the expression above, first with $c=d_1d_2$ and subsequently with $c=\gcd(N_1,N_2)$ yields,
\begin{equation*}
    d
    = \frac{\gcd(N_1d_2, N_2d_1)}{\gcd(N_1, N_2)} 
    = \frac{\gcd(N_1, N_2)}{\gcd(N_1, N_2)}\gcd\left(\frac{N_1d_2}{\gcd(N_1, N_2)}, \frac{N_2d_1}{\gcd(N_1, N_2)}\right)
    = \gcd(N_1'd_2, N_2'd_1) \ ,
\end{equation*}
where in the last equality we recall that $N_t' = N_t/\gcd(N_1,N_2)$.
Thus, B\'ezout's identity gives integers $p$ and $q$ such that
\begin{equation}
\label{eq:bezout}
    p(N_2'd_1) + q(N_1'd_2) = d \ .
\end{equation}

Since $d_1 = \gcd(k, N_1)$ and $d_2 = \gcd(l, N_2)$, $k/d_1$ and $l/d_2$ must have multiplicative inverses modulo $N_1$ and $N_2$, respectively.
Thus, set $m_j = j(k/d_1)^{-1}p \bmod{N_1}$ for the row index,
and set $n_j = h(l/d_2)^{-1}q \bmod{N_2}$ for the column index.
These values imply that 
\def\blank{}
\newcommand{\shrunkmod}[1]{\hspace{-6pt}\pmod{#1}}
\begin{alignat*}{6}
    m_j(k/d_1) &= jp \shrunkmod{N_1} \blank\implies\blank &m_jk &= jpd_1 &\shrunkmod{N_1} \blank\implies\blank &&m_jkN_2' &= jpd_1 N_2' &\shrunkmod{N_1N_2'} \ ,\\
    n_j(l/d_2) &= jq \shrunkmod{N_2} \blank\implies\blank &n_jl &= jqd_2 &\shrunkmod{N_2} \blank\implies\blank &&n_jlN_1' &= jqd_2 N_1' &\shrunkmod{N_2N_1'} \ .
\end{alignat*}
Since $N = N_1N_2/\gcd(N_1,N_2) = N_1'N_2 = N_1N_2'$, we can substitute these expressions into the left hand side of \cref{eq:to_satisfy} to obtain
\begin{equation*}
    m_jkN_2' + n_jlN_1' = jpd_1N_2' + jqd_2N_1' = j(pd_1N_2' + qd_2N_1') = jd \pmod{N} \ ,
\end{equation*}
where we have used \cref{eq:bezout} in the last equality.  This shows that $(m_j, n_j)$ satisfies \cref{eq:to_satisfy}

We next show that the same number of entries of ${\tt X}$ are summed for each choice of $j$.
Equivalently,
for any value of $j$,
the same number  of pairs $(m, n)$ satisfy \cref{eq:to_satisfy}.
Let 
$H = \set{(m, n) \in \mathbb{Z}_{N_1} \times \mathbb{Z}_{N_2} : mkN_2' + nlN_1' = 0 \pmod{N}}$,
and suppose $(m, n), (m', n') \in H$.
Since $(0, 0) \in H$, and
\begin{equation*}
    (m - m')kN_2' + (n - n')lN_1' = 0 \pmod{N}\ ,
\end{equation*}
the set $H$ is a subgroup of $\mathbb{Z}_{N_1} \times \mathbb{Z}_{N_2}$.
For any fixed $j$, since 
\begin{equation*}
    (m_j + m)kN_2' + (n_j + n)lN_1' = jd \pmod{N}
    \iff mkN_2' + nlN_1' = 0 \pmod{N} \ ,
\end{equation*}
the set of indices which satisfy the congruence in \cref{eq:to_satisfy} is given by the coset,
\begin{equation*}
    \set{(m, n) \in \mathbb{Z}_{N_1} \times \mathbb{Z}_{N_2} : mkN_2' + nlN_1' = jd \pmod{N}} = (m_j, n_j) + H\ .
\end{equation*}
Since the set of cosets partitions the group evenly and there are $D$ cosets of $H$ (as there are $D$ choices of $j$), we have shown that each entry of the subsignal is a sum of $|H| = N_1N_2/D$ entries of ${\tt X}$.
\end{proof}

%


\bibliographystyle{siamplain}
\bibliography{bib1}

@article{levinson2023,
  author = {Levinson, Howard W. and Markel, Vadim and Triantafillou, Nicholas},
  title={Inversion of band-limited discrete fourier transforms of binary images: Uniqueness and algorithms},
  journal={SIAM Journal on Imaging Sciences},
  volume={16},
  number={3},
  pages={1338--1369},
  year={2023},
  publisher={SIAM},
}

@ARTICLE{levinson2021,
  author={Levinson, Howard W. and Markel, Vadim A.},
  journal={IEEE Transactions on Signal Processing}, 
  title={Binary Discrete Fourier Transform and Its Inversion}, 
  year={2021},
  volume={69},
  number={},
  pages={3484-3499},
  doi={10.1109/TSP.2021.3088215},
}

@inproceedings{pei2023binary,
  title={Binary Image Fast Perfect Recovery from Sparse 2D-DFT Coefficients},
  author={Pei, Soo-Chang and Chang, Kuo-Wei},
  booktitle={ICASSP 2023-2023 IEEE International Conference on Acoustics, Speech and Signal Processing (ICASSP)},
  pages={1--5},
  year={2023},
  organization={IEEE},
}

@ARTICLE{pei2022binary1D,
  author={Pei, Soo-Chang and Chang, Kuo-Wei},
  journal={IEEE Transactions on Signal Processing}, 
  title={Binary Signal Perfect Recovery From Partial DFT Coefficients}, 
  year={2022},
  volume={70},
  number={},
  pages={3848-3861},
  doi={10.1109/TSP.2022.3190615}}

@article{hampejs2014representingcountingsubgroupsgroup,
author = {Hampejs, Mario and Holighaus, Nicki and Tóth, László and Wiesmeyr, Christoph},
title = {Representing and Counting the Subgroups of the Group $\mathbb{Z}_m \times\mathbb{Z}_n$},
journal = {Journal of Numbers},
volume = {2014},
number = {1},
pages = {491428},
doi = {https://doi.org/10.1155/2014/491428},
url = {https://onlinelibrary.wiley.com/doi/abs/10.1155/2014/491428},
eprint = {https://onlinelibrary.wiley.com/doi/pdf/10.1155/2014/491428},
year = {2014}
}

@article{Lenstra1982,
  title = {Factoring polynomials with rational coefficients},
  volume = {261},
  ISSN = {1432-1807},
  url = {http://dx.doi.org/10.1007/BF01457454},
  DOI = {10.1007/bf01457454},
  number = {4},
  journal = {Mathematische Annalen},
  publisher = {Springer Science and Business Media LLC},
  author = {Lenstra,  A. K. and Lenstra,  H. W. and Lovász,  L.},
  year = {1982},
  month = dec,
  pages = {515–534},
}

@misc{kalbach2024lllalgorithmlatticebasis,
      title={{LLL} Algorithm for Lattice Basis Reduction}, 
      author={Alex Kalbach and Ted Chinburg},
      year={2024},
      eprint={2410.22196},
      archivePrefix={arXiv},
      primaryClass={math.NT},
      url={https://arxiv.org/abs/2410.22196}, 
}

@inproceedings{nguyen2006average, 
    author = {Nguyen, Phong Q. and Stehl\'{e}, Damien}, 
    title = {{LLL} on the average}, 
    year = {2006}, 
    isbn = {3540360751}, 
    publisher = {Springer-Verlag}, 
    address = {Berlin, Heidelberg}, 
    url = {https://doi.org/10.1007/11792086_18}, 
    doi = {10.1007/11792086_18}, 
    booktitle = {Proceedings of the 7th International Conference on Algorithmic Number Theory}, 
    pages = {238–256}, 
    numpages = {19}, 
    location = {Berlin, Germany}, 
    series = {ANTS'06} 
}

@article{nguyen2009lll,
  title={An {LLL} algorithm with quadratic complexity},
  author={Nguyen, Phong Q and Stehl{\'e}, Damien},
  journal={SIAM Journal on Computing},
  volume={39},
  number={3},
  pages={874--903},
  year={2009},
  publisher={SIAM},
}

@incollection{stehle2009floating,
  title={Floating-point {LLL}: theoretical and practical aspects},
  author={Stehl{\'e}, Damien},
  booktitle={The {LLL} Algorithm: survey and applications},
  pages={179--213},
  year={2009},
  publisher={Springer},
}

@article{Lenstra1983,
 ISSN = {0364765X, 15265471},
 URL = {http://www.jstor.org/stable/3689168},
 author = {H. W. Lenstra},
 journal = {Mathematics of Operations Research},
 number = {4},
 pages = {538--548},
 publisher = {INFORMS},
 title = {Integer Programming with a Fixed Number of Variables},
 urldate = {2025-09-23},
 volume = {8},
 year = {1983},
}

@article{Aardal2000,
 ISSN = {0364765X, 15265471},
 URL = {http://www.jstor.org/stable/3690477},
 author = {Karen Aardal and Cor A. J. Hurkens and Arjen K. Lenstra},
 journal = {Mathematics of Operations Research},
 number = {3},
 pages = {427--442},
 publisher = {INFORMS},
 title = {Solving a System of Linear Diophantine Equations with Lower and Upper Bounds on the Variables},
 urldate = {2025-09-23},
 volume = {25},
 year = {2000},
}

@article{hastad1989,
author = {Hastad, J. and Just, B. and Lagarias, J. C. and Schnorr, C. P.},
title = {Polynomial Time Algorithms for Finding Integer Relations among Real Numbers},
journal = {SIAM Journal on Computing},
volume = {18},
number = {5},
pages = {859--881},
year = {1989},
doi = {10.1137/0218059},
URL = {https://doi.org/10.1137/0218059},
eprint = {https://doi.org/10.1137/0218059},
}

@techreport{bailey2009,
  author       = {Bailey, David H and Borwein, J M},
  title        = {PSLQ: An Algorithm to Discover Integer Relations},
  institution  = {Ernest Orlando Lawrence Berkeley National Laboratory, Berkeley, CA (US)},
  doi          = {10.2172/963658},
  url          = {https://www.osti.gov/biblio/963658},
  place        = {United States},
  year         = {2009},
  month        = {04}}

@InProceedings{Coppersmith1996,
author={Coppersmith, Don},
editor={Maurer, Ueli},
title={Finding a Small Root of a Univariate Modular Equation},
booktitle={Advances in Cryptology --- EUROCRYPT '96},
year={1996},
publisher={Springer Berlin Heidelberg},
address={Berlin, Heidelberg},
pages={155--165},
}

@ARTICLE{Zenzo1996,
  author={Di Zenzo, S. and Cinque, L. and Levialdi, S.},
  journal={IEEE Transactions on Pattern Analysis and Machine Intelligence}, 
  title={Run-based algorithms for binary image analysis and processing}, 
  year={1996},
  volume={18},
  number={1},
  pages={83-89},
  doi={10.1109/34.476016}}

@book{machand-maillet2000,
  author = {Stéphane Marchand-Maillet and Yazid M. Sharaiha},
  title = {Binary Digital Image Processing},
  ISBN = {9780124705050},
  url = {http://dx.doi.org/10.1016/B978-0-12-470505-0.X5000-X},
  DOI = {10.1016/b978-0-12-470505-0.x5000-x},
  publisher = {Elsevier},
  year = {2000}
}

@article{REN2002,
title = {Tracing boundary contours in a binary image},
journal = {Image and Vision Computing},
volume = {20},
number = {2},
pages = {125-131},
year = {2002},
issn = {0262-8856},
doi = {https://doi.org/10.1016/S0262-8856(01)00091-9},
url = {https://www.sciencedirect.com/science/article/pii/S0262885601000919},
author = {Mingwu Ren and Jingyu Yang and Han Sun},
}

@article{Cooley1965,
 ISSN = {00255718, 10886842},
 URL = {http://www.jstor.org/stable/2003354},
 author = {James W. Cooley and John W. Tukey},
 journal = {Mathematics of Computation},
 number = {90},
 pages = {297--301},
 publisher = {American Mathematical Society},
 title = {An Algorithm for the Machine Calculation of Complex Fourier Series},
 urldate = {2025-09-23},
 volume = {19},
 year = {1965},
}

@article{grigoryan-superposition,
author = {Grigoryan, Artyom and Du, Nan},
year = {2011},
month = {10},
pages = {2531 - 2541},
title = {Principle of Superposition by Direction Images},
volume = {20},
journal = {Image Processing, IEEE Transactions on},
doi = {10.1109/TIP.2011.2128334}
}

@book{grigoryan-book,
author = {Grigoryan, Artyom},
year = {2003},
month = {08},
pages = {},
title = {Multidimensional Discrete Unitary Transforms: Representation, Partitioning, and Algorithms},
isbn = {ISBN: 0-8247-4596-5},
doi = {10.1201/9781482276329},
publisher = {Marcel Dekker, Inc.}
}

@article{gertner2002new,
  title={A new efficient algorithm to compute the two-dimensional discrete Fourier transform},
  author={Gertner, Izidor},
  journal={IEEE Transactions on Acoustics, Speech, and Signal Processing},
  volume={36},
  number={7},
  pages={1036--1050},
  year={2002},
  publisher={IEEE}
}

@article{hsung1996discrete,
  title={The discrete periodic Radon transform},
  author={Hsung, TaiChiu and Lun, Daniel Pak-Kong and Siu, Wan-Chi},
  journal={IEEE Transactions on Signal Processing},
  volume={44},
  number={10},
  pages={2651--2657},
  year={1996},
  publisher={IEEE}
}

@article{kingston2007generalised,
  title={Generalised finite radon transform for N$\times$ N images},
  author={Kingston, Andrew and Svalbe, Imants},
  journal={Image and Vision Computing},
  volume={25},
  number={10},
  pages={1620--1630},
  year={2007},
  publisher={Elsevier}
}

@unpublished{intvert,
    author = {Isaac Viviano},
    title = {{intvert, a Python package for inversion of integer arrays from partial DFT samples}},
    year = 2025,
    note = {Available at \url{https://pypi.org/project/intvert/}},
    url = {https://pypi.org/project/intvert/},
}

@article{myerson1986small,
  title={How small can a sum of roots of unity be?},
  author={Myerson, Gerald},
  journal={The American Mathematical Monthly},
  volume={93},
  number={6},
  pages={457--459},
  year={1986},
  publisher={Taylor \& Francis},
}

@article{barber2023small,
  title={Small sums of five roots of unity},
  author={Barber, Ben},
  journal={Bulletin of the London Mathematical Society},
  volume={55},
  number={4},
  pages={1890--1906},
  year={2023},
  publisher={Wiley Online Library},
}

@ARTICLE{Buhler2000dense,
  author={Buhler, J. and Shokrollahi, M.A. and Stemann, V.},
  journal={IEEE Transactions on Information Theory}, 
  title={Fast and precise Fourier transforms}, 
  year={2000},
  volume={46},
  number={1},
  pages={213-228},
  doi={10.1109/18.817519}}

@article{karp1975computational,
  title={On the computational complexity of combinatorial problems},
  author={Karp, Richard M},
  journal={Networks},
  volume={5},
  number={1},
  pages={45--68},
  year={1975},
  publisher={Wiley Online Library},
}

@misc{gurobi,
  author = {{Gurobi Optimization, LLC}},
  title = {{Gurobi Optimizer Reference Manual}},
  year = 2026,
  url = "https://www.gurobi.com"
}

@article{padberg1991,
author = {Padberg, Manfred and Rinaldi, Giovanni},
title = {A Branch-and-Cut Algorithm for the Resolution of Large-Scale Symmetric Traveling Salesman Problems},
journal = {SIAM Review},
volume = {33},
number = {1},
pages = {60-100},
year = {1991},
doi = {10.1137/1033004},
URL = {https://doi.org/10.1137/1033004},
eprint = {https://doi.org/10.1137/1033004},
}

@article{plonka2016deterministic,
  title={A deterministic sparse FFT algorithm for vectors with small support},
  author={Plonka, Gerlind and Wannenwetsch, Katrin},
  journal={Numerical Algorithms},
  volume={71},
  number={4},
  pages={889--905},
  year={2016},
  publisher={Springer},
}

@article{rajaby2022structured,
  title={A structured review of sparse fast Fourier transform algorithms},
  author={Rajaby, Elias and Sayedi, Sayed Masoud},
  journal={Digital Signal Processing},
  volume={123},
  pages={103403},
  year={2022},
  publisher={Elsevier},
}

@article{bittens2019deterministic,
  title={A deterministic sparse FFT for functions with structured Fourier sparsity},
  author={Bittens, Sina and Zhang, Ruochuan and Iwen, Mark A},
  journal={Advances in Computational Mathematics},
  volume={45},
  number={2},
  pages={519--561},
  year={2019},
  publisher={Springer},
}

@inproceedings{hsieh2013sparse,
  title={Sparse fast fourier transform by downsampling},
  author={Hsieh, Sung-Hsien and Lu, Chun-Shien and Pei, Soo-Chang},
  booktitle={2013 IEEE International Conference on Acoustics, Speech and Signal Processing},
  pages={5637--5641},
  year={2013},
  organization={IEEE},
}

@article{donoho2006compressed,
  title={Compressed sensing},
  author={Donoho, David L},
  journal={IEEE Transactions on information theory},
  volume={52},
  number={4},
  pages={1289--1306},
  year={2006},
  publisher={IEEE},
}

@article{duarte2011structured,
  title={Structured compressed sensing: From theory to applications},
  author={Duarte, Marco F and Eldar, Yonina C},
  journal={IEEE Transactions on signal processing},
  volume={59},
  number={9},
  pages={4053--4085},
  year={2011},
  publisher={IEEE},
}

@inproceedings{gilbert2002near,
  title={Near-optimal sparse Fourier representations via sampling},
  author={Gilbert, Anna C and Guha, Sudipto and Indyk, Piotr and Muthukrishnan, Shanmugavelayutham and Strauss, Martin},
  booktitle={Proceedings of the thiry-fourth annual ACM symposium on Theory of computing},
  pages={152--161},
  year={2002},
}

@inproceedings{gilbert2005improved,
  title={Improved time bounds for near-optimal sparse Fourier representations},
  author={Gilbert, Anna C and Muthukrishnan, Shan and Strauss, Martin},
  booktitle={Wavelets xi},
  volume={5914},
  pages={398--412},
  year={2005},
  organization={SPIE},
}

@article{gilbert2014recent,
  title={Recent developments in the sparse Fourier transform: A compressed Fourier transform for big data},
  author={Gilbert, Anna C and Indyk, Piotr and Iwen, Mark and Schmidt, Ludwig},
  journal={IEEE Signal Processing Magazine},
  volume={31},
  number={5},
  pages={91--100},
  year={2014},
  publisher={IEEE},
}

@article{rudelson2008sparse,
  title={On sparse reconstruction from Fourier and Gaussian measurements},
  author={Rudelson, Mark and Vershynin, Roman},
  journal={Communications on Pure and Applied Mathematics: A Journal Issued by the Courant Institute of Mathematical Sciences},
  volume={61},
  number={8},
  pages={1025--1045},
  year={2008},
  publisher={Wiley Online Library},
}

@inproceedings{rudelson2006sparse,
  title={Sparse reconstruction by convex relaxation: Fourier and Gaussian measurements},
  author={Rudelson, Mark and Vershynin, Roman},
  booktitle={2006 40th Annual Conference on Information Sciences and Systems},
  pages={207--212},
  year={2006},
  organization={IEEE},
}

@article{candes2006stable,
  title={Stable signal recovery from incomplete and inaccurate measurements},
  author={Candes, Emmanuel J and Romberg, Justin K and Tao, Terence},
  journal={Communications on Pure and Applied Mathematics: A Journal Issued by the Courant Institute of Mathematical Sciences},
  volume={59},
  number={8},
  pages={1207--1223},
  year={2006},
  publisher={Wiley Online Library},
}

@ARTICLE{candes2006robust,
  author={Candes, E.J. and Romberg, J. and Tao, T.},
  journal={IEEE Transactions on Information Theory}, 
  title={Robust uncertainty principles: exact signal reconstruction from highly incomplete frequency information}, 
  year={2006},
  volume={52},
  number={2},
  pages={489-509},
  doi={10.1109/TIT.2005.862083}}

@book{herman2012discrete,
  title={Discrete tomography: Foundations, algorithms, and applications},
  author={Herman, Gabor T and Kuba, Attila},
  year={2012},
  publisher={Springer Science \& Business Media},
}

@article{YAGLE2001,
title = {A convergent composite mapping Fourier domain iterative algorithm for 3-D discrete tomography},
journal = {Linear Algebra and its Applications},
volume = {339},
number = {1},
pages = {91-109},
year = {2001},
issn = {0024-3795},
doi = {https://doi.org/10.1016/S0024-3795(01)00458-X},
url = {https://www.sciencedirect.com/science/article/pii/S002437950100458X},
author = {Andrew E. Yagle},
}

@article{HAJDU2005,
title = {Unique reconstruction of bounded sets in discrete tomography},
journal = {Electronic Notes in Discrete Mathematics},
volume = {20},
pages = {15-25},
year = {2005},
note = {Proceedings of the Workshop on Discrete Tomography and its Applications},
issn = {1571-0653},
doi = {https://doi.org/10.1016/j.endm.2005.04.002},
url = {https://www.sciencedirect.com/science/article/pii/S1571065305050560},
author = {Lajos Hajdu},
}

@Inbook{lungo1999tomography,
author="Lungo, Alberto Del
and Nivat, Maurice",
editor="Herman, Gabor T.
and Kuba, Attila",
title="Reconstruction of Connected Sets from Two Projections",
bookTitle="Discrete Tomography: Foundations, Algorithms, and Applications",
year="1999",
publisher="Birkh{\"a}user Boston",
address="Boston, MA",
pages="163--188",
isbn="978-1-4612-1568-4",
doi="10.1007/978-1-4612-1568-4_7",
url="https://doi.org/10.1007/978-1-4612-1568-4_7"
}

@Inbook{gardner1999tomography,
author="Gardner, Richard J.
and Gritzmann, Peter",
editor="Herman, Gabor T.
and Kuba, Attila",
title="Uniqueness and Complexity in Discrete Tomography",
bookTitle="Discrete Tomography: Foundations, Algorithms, and Applications",
year="1999",
publisher="Birkh{\"a}user Boston",
address="Boston, MA",
pages="85--113",
isbn="978-1-4612-1568-4",
doi="10.1007/978-1-4612-1568-4_4",
url="https://doi.org/10.1007/978-1-4612-1568-4_4"
}

@article{hajdu2001algebraic,
  title={Algebraic aspects of discrete tomography},
  author={Hajdu, Lajos and Tijdeman, Rob},
  journal={Journal fur die Reine und Angewandte Mathematik},
  volume={534},
  pages={119--128},
  year={2001},
  publisher={Berlin, W. de Gruyter.}
}

@inproceedings{guedon1995psychovisual,
  title={Psychovisual image coding via an exact discrete Radon transform},
  author={Guedon, Jeanpierre V and Barba, Dominique and Burger, Nicole},
  booktitle={Visual Communications and Image Processing'95},
  volume={2501},
  pages={562--572},
  year={1995},
  organization={SPIE}
}

@InProceedings{guedon2005mojette,
author="Gu{\'e}don, JeanPierre
and Normand, Nicolas",
editor="Andres, Eric
and Damiand, Guillaume
and Lienhardt, Pascal",
title="The Mojette Transform: The First Ten Years",
booktitle="Discrete Geometry for Computer Imagery",
year="2005",
publisher="Springer Berlin Heidelberg",
address="Berlin, Heidelberg",
pages="79--91",
isbn="978-3-540-31965-8"
}

@article{chrobak1999convex,
title = {Reconstructing hv-convex polyominoes from orthogonal projections},
journal = {Information Processing Letters},
volume = {69},
number = {6},
pages = {283-289},
year = {1999},
issn = {0020-0190},
doi = {https://doi.org/10.1016/S0020-0190(99)00025-3},
url = {https://www.sciencedirect.com/science/article/pii/S0020019099000253},
author = {Marek Chrobak and Christoph Dürr}
}

@inproceedings{brunetti2000reconstruction,
  title={Reconstruction of discrete sets from two or more X-rays in any direction},
  author={Brunetti, Sara and Daurat, Alain},
  booktitle={Seventh International Workshop on Combinatorial Image Analysis (IWCIA'00)-Theory and Applications-},
  pages={241--258},
  year={2000},
  organization={Universit{\'e} de Caen}
}

@article{BRUNETTI20132281,
title = {Discrete tomography determination of bounded lattice sets from four X-rays},
journal = {Discrete Applied Mathematics},
volume = {161},
number = {15},
pages = {2281-2292},
year = {2013},
note = {Advances in Discrete Geometry: 16th International Conference on Discrete Geometry for Computer Imagery},
issn = {0166-218X},
doi = {https://doi.org/10.1016/j.dam.2012.09.010},
url = {https://www.sciencedirect.com/science/article/pii/S0166218X12003599},
author = {S. Brunetti and P. Dulio and C. Peri},
}

@book{parker2014discrete,
  title={Discrete optimization},
  author={Parker, R Gary and Rardin, Ronald L},
  year={2014},
  publisher={Elsevier},
}

@inproceedings{boufounos20081,
  title={1-bit compressive sensing},
  author={Boufounos, Petros T and Baraniuk, Richard G},
  booktitle={2008 42nd Annual Conference on Information Sciences and Systems},
  pages={16--21},
  year={2008},
  organization={IEEE},
}

@article{kadu2019convex,
  title={A convex formulation for binary tomography},
  author={Kadu, Ajinkya and van Leeuwen, Tristan},
  journal={IEEE Transactions on Computational Imaging},
  volume={6},
  pages={1--11},
  year={2019},
  publisher={IEEE},
}

@article{batenburg2011dart,
  title={DART: a practical reconstruction algorithm for discrete tomography},
  author={Batenburg, Kees Joost and Sijbers, Jan},
  journal={IEEE Transactions on Image Processing},
  volume={20},
  number={9},
  pages={2542--2553},
  year={2011},
  publisher={IEEE},
}

@article{batenburg20093d,
  title={3D imaging of nanomaterials by discrete tomography},
  author={Batenburg, Kees Joost and Bals, Sara and Sijbers, Jan and K{\"u}bel, Christian and Midgley, Paul Anthony and Hernandez, JC and Kaiser, Ute and Encina, Ezequiel R and Coronado, Eduardo A and Van Tendeloo, Gustaaf},
  journal={Ultramicroscopy},
  volume={109},
  number={6},
  pages={730--740},
  year={2009},
  publisher={Elsevier},
}

@article{joseph2025low,
  title={Low-resolution compressed sensing and beyond for communications and sensing: Trends and opportunities},
  author={Joseph, Geethu and Gandikota, Venkata and Bhandari, Ayush and Choi, Junil and Kim, In-soo and Lee, Gyoseung and Matthaiou, Michail and Murthy, Chandra R and Ngo, Hien Quoc and Varshney, Pramod K and others},
  journal={Signal Processing},
  volume={235},
  pages={110020},
  year={2025},
  publisher={Elsevier},
}

@article{zhan2014integer,
  title={Integer-forcing linear receivers},
  author={Zhan, Jiening and Nazer, Bobak and Erez, Uri and Gastpar, Michael},
  journal={IEEE Transactions on Information Theory},
  volume={60},
  number={12},
  pages={7661--7685},
  year={2014},
  publisher={IEEE},
}

@article{herman2003discrete,
  title={Discrete tomography in medical imaging},
  author={Herman, Gabor T and Kuba, Attila},
  journal={Proceedings of the IEEE},
  volume={91},
  number={10},
  pages={1612--1626},
  year={2003},
  publisher={IEEE},
}

@incollection{alpers2002stability,
  title={Stability and instability in discrete tomography},
  author={Alpers, Andreas and Gritzmann, Peter and Thorens, Lionel},
  booktitle={Digital and Image Geometry: Advanced Lectures},
  pages={175--186},
  year={2002},
  publisher={Springer},
}

@article{mao2012reconstruction,
  title={Reconstruction of binary functions and shapes from incomplete frequency information},
  author={Mao, Yu},
  journal={IEEE Transactions on Information Theory},
  volume={58},
  number={6},
  pages={3642--3653},
  year={2012},
  publisher={IEEE},
}

@article{nashold2002synthesis,
  title={Synthesis of two-dimensional binary images through band-limited systems: a slicing method},
  author={Nashold, Karen M and Saleh, Bahaa EA},
  journal={IEEE Transactions on Acoustics, Speech, and Signal Processing},
  volume={37},
  number={8},
  pages={1271--1279},
  year={2002},
  publisher={IEEE},
}

@article{nashold1989synthesis,
  title={Synthesis of binary images from band-limited functions},
  author={Nashold, Karen M and Bucklew, James A and Rudin, Walter and Saleh, Bahaa EA},
  journal={Journal of the Optical Society of America A},
  volume={6},
  number={6},
  pages={852--858},
  year={1989},
  publisher={Optical Society of America},
}

@article{vetterli2002sampling,
  title={Sampling signals with finite rate of innovation},
  author={Vetterli, Martin and Marziliano, Pina and Blu, Thierry},
  journal={IEEE transactions on Signal Processing},
  volume={50},
  number={6},
  pages={1417--1428},
  year={2002},
  publisher={IEEE},
}

@article{esedoglu2003blind,
  title={Blind deconvolution of bar code signals},
  author={Esedoglu, Selim},
  journal={Inverse Problems},
  volume={20},
  number={1},
  pages={121},
  year={2003},
  publisher={IOP Publishing},
}

@article{stolk2010algebraic,
  title={An algebraic framework for discrete tomography: Revealing the structure of dependencies},
  author={Stolk, Arjen and Batenburg, K Joost},
  journal={SIAM Journal on Discrete Mathematics},
  volume={24},
  number={3},
  pages={1056--1079},
  year={2010},
  publisher={SIAM},
}

@article{tropp2008linear,
  title={On the linear independence of spikes and sines},
  author={Tropp, Joel A},
  journal={Journal of Fourier Analysis and Applications},
  volume={14},
  number={5},
  pages={838--858},
  year={2008},
  publisher={Springer},
}

@Article{tao_2005_1,
  title =        "An uncertainty principle for cyclic groups of prime order",
  author =       "Tao, T.",
  journal =      "Math. Res. Lett.",
  year =         2005,
  volume =       12,
  pages =        "121-127",
}

@article{schreiber2005image,
  title={Image processing for quality improvement},
  author={Schreiber, William F},
  journal={Proceedings of the IEEE},
  volume={66},
  number={12},
  pages={1640--1651},
  year={2005},
  publisher={IEEE}
}

@INPROCEEDINGS{tiwari2016,
  author={Tiwari, Sumit},
  booktitle={2016 International Conference on Information Technology (ICIT)}, 
  title={An Introduction to QR Code Technology}, 
  year={2016},
  volume={},
  number={},
  pages={39-44},
  doi={10.1109/ICIT.2016.021}}

@inproceedings{ajtai1998shortest,
  title={The shortest vector problem in L2 is NP-hard for randomized reductions},
  author={Ajtai, Mikl{\'o}s},
  booktitle={Proceedings of the thirtieth annual ACM symposium on Theory of computing},
  pages={10--19},
  year={1998}
}

@article{micciancio2001shortest,
  title={The shortest vector in a lattice is hard to approximate to within some constant},
  author={Micciancio, Daniele},
  journal={SIAM journal on Computing},
  volume={30},
  number={6},
  pages={2008--2035},
  year={2001},
  publisher={SIAM}
}

\end{document}